\documentclass[a4paper,english,reqno, 10pt]{scrartcl}
\usepackage[utf8]{inputenc}

\usepackage{array}

\usepackage{graphicx}

\usepackage{amsmath, amsthm, amssymb}
\usepackage{mathtools}

\usepackage{xcolor}

\usepackage[margin=2cm]{geometry}

\renewcommand\le\leqslant
\renewcommand\ge\geqslant

\newcommand{\fg}[0]{\mathfrak{g}}

\newcommand{\fs}[0]{\mathfrak{s}}

\newcommand{\kf}{\kappa}

\newcommand{\add}{\dotplus}
\newcommand{\ot}{\otimes}

\DeclareMathOperator{\tOp}{t}

\DeclareMathOperator{\adOp}{ad}
\newcommand{\ad}{\mathord{\adOp}}

\usepackage{nameref}
\usepackage{hyperref}
\usepackage[capitalise, nameinlink]{cleveref}

\usepackage{enumitem}
\setlist[itemize]{align=parleft,left=0pt..18pt,topsep=5pt,itemsep=2pt}
\setlist[enumerate]{align=parleft,left=0pt..18pt,topsep=5pt, itemsep=2pt}

\usepackage[backend=bibtex, style=numeric, maxbibnames=99]{biblatex} 
\newtheorem{thm}{Theorem}[section]
\theoremstyle{plain}

\newtheorem{corollary}[thm]{Corollary}
\newtheorem{lemma}[thm]{Lemma}
\newtheorem{proposition}[thm]{Proposition}
\newtheorem{theorem}[thm]{Theorem}

\newtheorem{maintheoremcounter}{Theorem}
\theoremstyle{plain}

\newtheorem{maincorollary}[maintheoremcounter]{Corollary}
\newtheorem{mainproposition}[maintheoremcounter]{Proposition}
\newtheorem{maintheorem}[maintheoremcounter]{Theorem}

\newtheoremstyle{example_style}
  {5pt} 
  {5pt} 
  {} 
  {} 
  {\itshape} 
  {.} 
  {10pt} 
  {} 

\theoremstyle{example_style}
\newtheorem{examplex}[thm]{Example}
\newenvironment{example}
  {\pushQED{\qed}\examplex}
  {\popQED\endexamplex}
  
  \crefname{examplex}{Example}{Examples}

\newtheorem{remarkx}[thm]{Remark}
\newenvironment{remark}
  {\pushQED{\qed}\remarkx}
 {\popQED\endremarkx}

\theoremstyle{definition}

\newtheorem{definitionx}[thm]{Definition}
\newenvironment{definition}
  {\pushQED{\qed}\definitionx}
 {\popQED\enddefinitionx}

\title{Topological Manin pairs and \((n,s)\)-type series}
\date{\today}

\author{
Raschid Abedin\footnote{Department of Mathematics, ETH Zürich, 8006 Zürich, Schweiz, raschid.abedin@math.ethz.ch.}, \, 
Stepan Maximov\footnote{Department of Mathematical Sciences, 
Chalmers University of Technology  and the University of Gothenburg, 412 96 Gothenburg, Sweden, maximov@chalmers.se.} \, and
Alexander Stolin\footnote{Department of Mathematical Sciences, 
Chalmers University of Technology  and the University of Gothenburg, 412 96 Gothenburg, Sweden, astolin@chalmers.se}}

\usepackage{lipsum}

\begin{document}
\maketitle

\begin{abstract}
Lie subalgebras of
$ L = \mathfrak{g}(\!(x)\!) \times \mathfrak{g}[x]/x^n\mathfrak{g}[x] $,
complementary to the diagonal
embedding $\Delta$ of $ \mathfrak{g}[\![x]\!] $ and Lagrangian
with respect to some particular form,
are in bijection with formal classical 
$r$-matrices and topological
Lie bialgebra structures on
the Lie algebra of formal power series
$ \mathfrak{g}[\![x]\!] $.
In this work we consider arbitrary
subspaces of $ L $ complementary to $\Delta$ and associate them with
so-called
series of type $ (n,s) $.

We prove that Lagrangian subspaces are in bijection with skew-symmetric
$ (n,s) $-type series and
topological quasi-Lie bialgebra
structures on $ \mathfrak{g}[\![x]\!] $. 
Using the classificaiton
of Manin pairs we classify up to
twisting and coordinate transformations
all quasi-Lie bialgebra structures.

Series of type $ (n,s) $,
solving the generalized classical Yang-Baxter
equation, correspond to subalgebras
of $L$.
We discuss their possible
utility in the theory of integrable
systems.
\end{abstract}

\begin{center}
\vspace{0.5cm}
    Dedicated to the memory of Yuri Manin
\vspace{0.5cm}
\end{center}

\section{Introduction}
Let \( F \) be an algebraically
closed field of characteristic \( 0 \)
equipped with the discrete topology
and \( \fg \) be a simple Lie algebra
over \( F \).
We define the Lie algebra \( \fg[\![x]\!] \)
to be the space
\( \fg \ot F[\![x]\!] \) with the bracket
\begin{equation}%
\label{eq:tensor_product_bracket}
    [a \ot f, b \ot g] = [a,b] \ot fg
\end{equation}
and we equip it with the \( (x) \)-adic
topology. The continuous dual of
\( \fg[\![x]\!] \) is denoted by
\( \fg[\![x]\!]'\) and it is 
endowed with the discrete topology.

A topological Manin pair
is a pair \( (L, \fg[\![x]\!]) \) where
\begin{enumerate}
    \item \( L \) is a Lie algebra equipped with an invariant non-degenerate
    symmetric bilinear form \( B \);
    \item \( \fg[\![x]\!] \subset L \) is a Lagrangian subalgebra with respect to \( B \); 
    \item for any continuous functional
    \( T \colon \fg[\![x]\!] \to F \) there is \( f \in L \) such that
    \( T = B(f, -) \).
\end{enumerate}
Topological Manin pairs were classified
in \cite{AMSZ} using the tools from
\cite{Stolin_Zelmanov_Montaner}.
More precisely, if 
\( (L, \fg[\![x]\!]) \)
is a topological Manin pair, then
\( L \) is isomorphic, 
as a Lie algebra
with form, to either
\( L(\infty) \)
or \( L(n, \alpha) \),
for some sequence
\( \alpha = (\alpha_i \in F \mid -\infty < i \le n-2) \) and an 
integer \( n \ge 0 \).
Here e.g.\ \( L(n, \alpha) \) is the Lie algebra \(\fg(\!(x)\!) \times \fg[x]/x^n\fg[x]\) equipped with a particular bilinear form defined by the sequence \( \alpha \). For exact definitions
see \cref{sec:top_manin_pairs}.

Let \( (L, \fg[\![x]\!]) \) be
a topological Manin pair.
Subspaces \( W \subset L \)
complementary to \( \fg[\![x]\!] \), i.e.\ 
\( \fg[\![x]\!] \add W = L \),
have interesting connections
to algebraic structures on
the Lie algebra \( \fg[\![x]\!] \) and solutions of the 
(generalized) classical Yang-Baxter
equation. This can be seen
from the following two examples.

\begin{example}%
\label{ex:Lagrangian_subalgebras_CYBE}
It was proven in \cite{AMSZ}
that topological Lie bialgebra
structures on \( \fg[\![x]\!] \)
are in one-to-one correspondence
with Lagrangian Lie subalgebras
of \( L(\infty) \) or
\( L(n, \alpha) \), \( 0 \le n \le 2 \), complementary to
\( \fg[\![x]\!] \).
Furthermore, such subspaces
are in bijection with formal
(classical)
\(r\)-matrices, i.e.\
series of the form
\begin{equation}%
\label{eq:classica_r_matr_intro}
    \frac{s(y)\Omega}{x-y} + g(x,y) = s(y) \Omega \sum_{k \ge 0} x^{-k - 1} y^{k} + g(x,y) \in (\fg \ot \fg)(\!(x)\!)[\![y]\!],
\end{equation}
where \( \Omega \in \fg \ot \fg \) is the quadratic Casimir element,  \( s(y) \in F[\![y]\!] \)
and \( g(x,y) \in (\fg \ot \fg)[\![x,y]\!] \), solving
the classical Yang-Baxter equation (CYBE). More precisely,
we have the following one-to-one
correspondences:

{\vspace{3mm}
\renewcommand{\arraystretch}{1.5}
    \centering
    \begin{tabular}{|m{6cm}|m{7.5cm}|}
    \hline
        Lagrangian subalgebras \(W \subset L(\infty)\), \newline
        \(W \add \fg[\![x]\!] = L(\infty)\) 
        &  
        Skew-symmetric series \( g(x,y) \in (\fg \ot \fg)[\![x,y]\!] \) solving CYBE
        \\
    \hline
         Lagrangian subalgebras \(W \subset L(n,\alpha)\), \newline
         \(W \add \fg[\![x]\!] = L(n,\alpha)\)
         &
         \(r\)-matrices \( \frac{s(y)\Omega}{x-y} + g(x,y) \), with \( s(y) \in y^nF[\![y]\!]^\times \)
        \\
        \hline
    \end{tabular} \par
}
\end{example}

\begin{example}%
\label{ex:subalgebras_GCYBE}
By e.g.\ \cite{skrypnyk_infinite_dimensional_Lie_algebras} (see \cite[Proposition 1.14]{abedin2021geometrization} for the formal case)
Lie subalgebras, not necessarily Lagrangian, \(W \subset L(0,0) \)
complementary to \( \fg[\![x]\!] \) are in bijection with
normalized formal generalized r-matrices, i.e.\
series of the form
\begin{equation}%
\label{eq:generalized_r_matr_intro}
    \frac{\Omega}{x-y} + g(x,y)
    = \Omega \sum_{k \ge 0} x^{-k - 1} y^{k} + g(x,y) \in (\fg \ot \fg)(\!(x)\!)[\![y]\!],
\end{equation}
where \( g(x,y) \in (\fg \ot \fg)[\![x]\!] \),
solving the generalized Yang-Baxter equation (GCYBE).
\end{example}

The proofs of the statements
given in
\cref{ex:Lagrangian_subalgebras_CYBE,ex:subalgebras_GCYBE}
lead to another viewpoint
on classical/generalized
\(r\)-matrices:
these objects are
generating series for
some specific subspaces of 
\( L(\infty) \) or \( L(n, \alpha) \),
\( 0 \le n \le 2 \), complementary
to \( \fg[\![x]\!] \).
In this paper we generalize and
develop this idea. 

We start by defining series of
type \( (n,s) \). 
Let us identify \( \fg[\![x]\!] \)
with the diagonal
\begin{equation}
    \Delta \coloneqq \{ (f, [f]) \mid f \in \fg[\![x]\!] \} \subset L(n,\alpha),
\end{equation}
and fix a basis \( \{b_i \}_{i=1}^{d} \)
of \( \fg \) orthonormal with respect
to its Killing form \( \kf \).
Instead of interpreting 
\( y^n\Omega/(x-y) \) as a series in
\( (\fg(\!(x)\!) \ot \fg )[\![y]\!] \)
we look at it as the series
\begin{equation}%
\label{eq:n_s_series_intro}
\begin{aligned}
    \frac{y^n\Omega}{x-y} = \sum_{k = 0}^{\infty}\sum_{i = 1}^d w_{k,i} \otimes b_i y^k \in \left(L(n, \alpha) \ot \fg \right)[\![y]\!].
\end{aligned}
\end{equation}
Elements \(w_{k,i} \in L(n,\alpha) = \fg(\!(x)\!) \times \fg[x]/x^n\fg[x]\) 
are presented explicitly in \cref{eq:def_wki}.
A series of \( (n,s) \)-type
is a series of the form
\begin{equation}
    \frac{s(x)y^n \Omega}{x-y} + g(x,y) \in 
    (L(n,\alpha) \ot \fg)[\![y]\!],
\end{equation}
where 
\( s \in F[\![x]\!]^\times \)
and 
\( g \in (\fg \ot \fg)[\![x,y]\!] \); See \cref{def:n_s_series}.
For each series \( r \) of type \( (n,s) \) we define
another series \( \overline{r} \) 
of the same type as follows
\begin{equation}
    \overline{r} = \frac{s(y)x^n\Omega}{x-y} - \tau(g(y,x)),
\end{equation}
where \( \tau \) is the \( F[\![x,y]\!] \)-linear
extension of the map \( a \ot b \mapsto b \ot a \).

The first main result of this paper
is that such series give a description
of subspaces \( W \subset L(n,\alpha) \)
complementary to \( \Delta \).
\begin{maintheorem}%
\label{mainthm:thmA}
Let \(n \in \mathbb{Z}_{\ge 0} \) 
and \( \alpha = (\alpha_i \in F \mid -\infty < i \le n-2) \) be an
arbitrary sequence with the corresponding series \( \alpha(x) \coloneqq x^{-n} + \alpha_{n-2}x^{-n+1} + \dots + \alpha_0 x^{-1} + \dots \in F(\!(x)\!) \).
For any \( (n,s) \)-type series
\begin{equation}
    r = \sum_{k = 0}^{\infty}\sum_{i = 1}^d f_{k,i} \otimes b_i y^k \in \left(L(n, \alpha) \ot \fg \right)[\![y]\!]
\end{equation}
define the space
\begin{equation}
    W(r) \coloneqq \textnormal{span}_F\{ f_{k,i} \mid k \ge 0, \, 1 \le i \le d \} \subseteq L(n,\alpha).
\end{equation}
The following results are true:
\begin{enumerate}
    \item \(W\) defines a bijection between series of type \(\left(n,\frac{1}{x^n \alpha(x)}\right)\) and subspaces \(V \subset L(n,\alpha)\)
    complementary to the diagonal
    \( \Delta \), i.e.\
    \( L(n,\alpha) = \Delta \add V;\)
    \item For any series \( r \) of type \(\left(n,\frac{1}{x^n \alpha(x)}\right)\)
    we have \( W(r)^\perp = W(\overline{r}) \) inside \( L(n, \alpha) \);

    \item Any series \(r\) of type \(\left(n,\frac{1}{x^n \alpha(x)}\right)\) satisfies \(\textnormal{GCYB}(r) = \psi\) (see \cref{def:gcybe} for the meaning of \(\textnormal{GCYB}(r)\)), where \(\psi \in (\fg \otimes \fg \otimes \fg)[\![x_1,x_2,x_3]\!]\) is defined by 
    \[B(v_1 \otimes v_2 \otimes v_3,\psi) = B(v_1,[v_2,v_3])\]
    for all \(v_1 \in W(\overline{r}),v_2,v_3 \in W(r)\).
\end{enumerate}
\end{maintheorem}


In particular, considering
the cases when
\( r \) is skew-symmetric
or \( \psi  = 0\)
we get the following correspondences.

\begin{maincorollary}%
\label{maincorollary}
    Let \( n \in \mathbb{Z}_{\ge 0} \), \( \alpha = (\alpha_i \in F \mid -\infty < i \le n-2) \)  and \( W \)
    be the map from \cref{mainthm:thmA}.
    Then
    \begin{enumerate}
        \item \( W \)
    defines a bijection between skew-symmtric
    \(\left(n,\frac{1}{x^n \alpha(x)}\right)\)-type series
    and Lagrangian subspaces \(V \subseteq L(n,\alpha)\) 
    complementary to the diagonal \( \Delta \);
    \item \(W\) defines a bijection between \(\left(n,\frac{1}{x^n \alpha(x)}\right)\)-type series solving GCYBE and subalgebras \(V \subseteq L(n,\alpha)\) complementary to the diagonal \( \Delta \).
    \end{enumerate}
\end{maincorollary}

The requirement on a series \( r \) of type
\( (n,s) \) to solve the CYBE is equivalent to
being skew-symmetric and to solve GCYBE.
Together with \cref{maincorollary}
this implies that Lagrangian subalgebras
\( W \subset L(n,\alpha) \) satisfying
\( W \add \Delta = L(n,\alpha) \)
are in bijection with \( (n,s) \)-type
series solving the classical Yang-Baxter equation.
These correspondences are schematically depicted
in \cref{fig:correspondence}.
\begin{figure}[h!]
    \centering
    \includegraphics[scale=0.35]{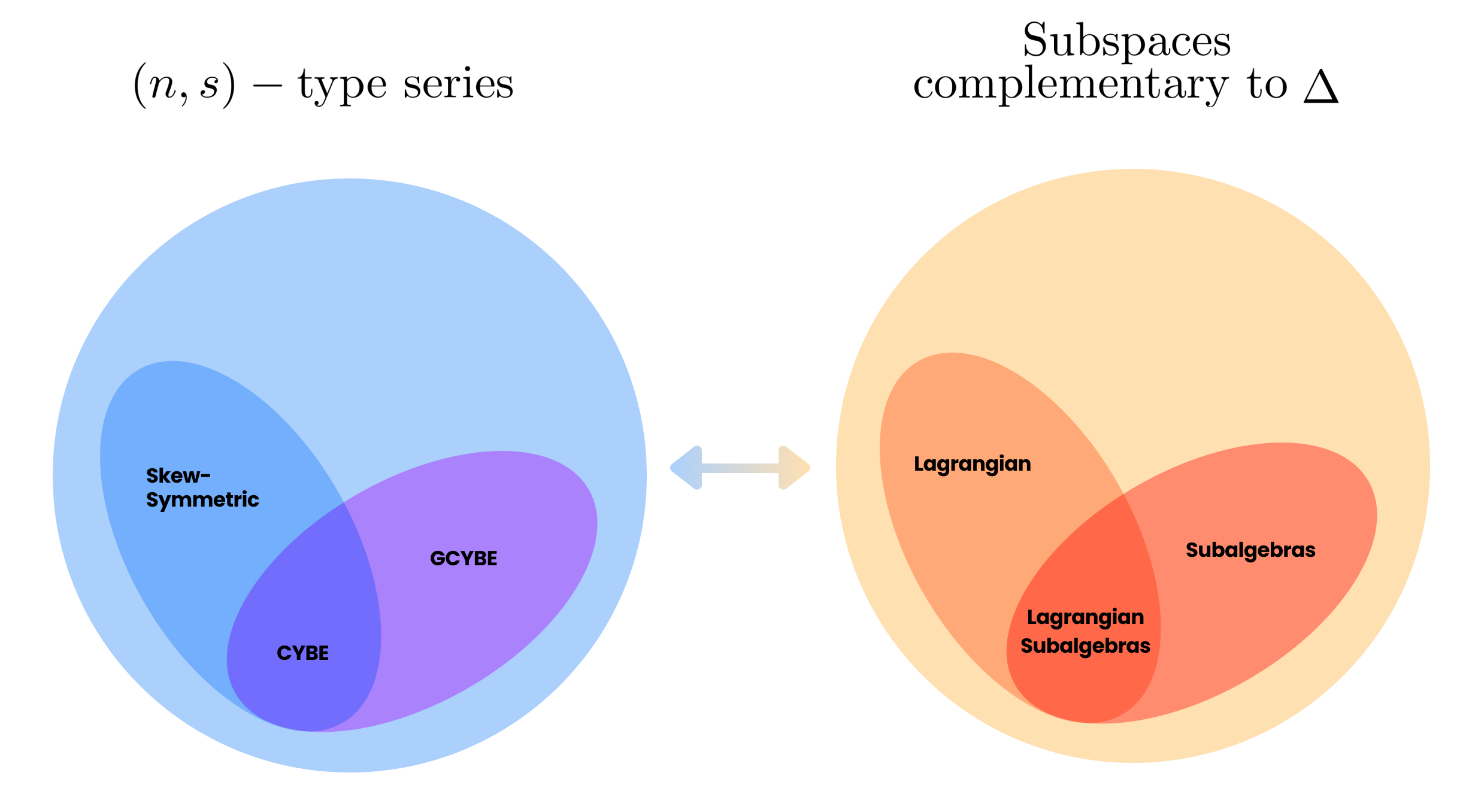}
    \caption{Series-subspaces correspondence}
    \label{fig:correspondence}
\end{figure}

The results above, at first glance, may look different
from the ones in \cref{ex:Lagrangian_subalgebras_CYBE,ex:subalgebras_GCYBE},
because \( (n,s) \)-type series do not live in the space
\( (\fg \ot \fg)(\!(x)\!)[\![y]\!] \).
However, if we start with a series \( R \) of type \( (n,s) \)
and simply reinterpret its singular part \( y^n\Omega/(x-y) \),
as it was done in \cref{eq:classica_r_matr_intro},
we obtain an element \(r\) in the space \( (\fg \ot \fg)(\!(x)\!)[\![y]\!] \).
And conversely, starting with an element \(r\) of the form
\cref{eq:classica_r_matr_intro} or \cref{eq:generalized_r_matr_intro}
and reinterpreting its singular part as an element of
\( (L(n,\alpha) \ot \fg)[\![y]\!] \)
we get a series \( R \) of type \( (n,s) \);
See \cref{rem:projections_vs_large_r}.
The first procedure is equivalent to the projection of
\(L(n,\alpha)\) onto its left component \( \fg(\!(x)\!) \)
and the inverse operation is equivalent to
taking two Taylor series expansions of \( r \)
at \( x = 0\) and \( y = 0\) respectively and then
constructing \(R\) by combining the coefficients of
\( b_i y^k \), \( k \ge 0 \), in these expansions.
The latter operation is exactly the tool that
was used in \cite[Section 5]{AMSZ}
to prove the relations presented in
\cref{ex:Lagrangian_subalgebras_CYBE}.
Moreover, by definition \( L(0,0) \cong \fg(\!(x)\!) \)
and hence the \( (0,0) \)-type series are precisely
the formal generalized \(r\)-matrices mentioned in
\cref{eq:generalized_r_matr_intro}.
Therefore, the statements of 
\cref{mainthm:thmA,maincorollary} indeed generalize 
and extend the examples above.

Reinterpreting the results of
\cite{AMSZ}
in terms of \( (n,s) \)-type series we see that
skew-symmetric series of type
\( \left(n, \frac{1}{x^n \alpha(x)}\right) \)
that also solve GCYBE exist only for
\( n = 0, 1\) and \( n = 2 \) with \( \alpha_0 = 0 \).

Lagrangian subalgebras of \( L(n,\alpha) \)
or \( L(\infty) \)
complementary to \( \Delta \) correspond
to topological Lie bialgebra structures on 
\( \fg[\![x]\!] \).
If we instead consider Lagrangian subspaces 
(not necessarily subalgebras)
of \( L(n,\alpha) \) or \( L(\infty) \), 
we get so called
topological quasi-Lie bialgebra structures
on \( \fg[\![x]\!] \).
A topological quasi-Lie bialgebra
structure on \( \fg[\![x]\!] \)
consists of
\begin{itemize}
    \item a skew-symmetric continuous linear map
    \( \delta \colon \fg[\![x]\!] \to (\fg \ot \fg)[\![x,y]\!] \) and
    
    \item a skew-symmetric element
    \( \varphi \in (\fg \ot \fg \ot \fg)[\![x,y,z]\!] \),
\end{itemize}
which are subject to the following three conditions
\begin{enumerate}
    \item \(\delta([a,b]) = [a \ot 1 + 1 \ot a, \delta(b)] - [b \ot 1 + 1 \ot b , \delta (a)]\), i.e.\  \( \delta \) is a \( 1 \)-cocycle;
    
    \item \(\frac{1}{2}\textnormal{Alt}((\delta \ot 1)\delta(a)) = 
    [a \ot 1 \ot 1 + 1 \ot a \ot 1 + 1 \ot 1 \ot a, \varphi]\);
    
    \item \( \textnormal{Alt}((\delta \ot 1 \ot 1)\varphi) = 0 \),
\end{enumerate}
where \( \textnormal{Alt}(x_1 \ot \dots \ot x_n) \coloneqq \sum_{\sigma \in S_n} \textnormal{sgn}(\sigma)x_{\sigma(1)} \ot \dots \ot x_{\sigma(n)} \).

Following
\cite{alekseev_kosmann}
we prove the following direct relation between
\( \delta \), \( \varphi \) and skew-symmetric \((n,s)\)-type series \( r \).
\begin{mainproposition}\label{mainprop:B}
There is a bijection between topological quasi-Lie bialgebras 
and skew-symmetric \((n,s)\)-type series.
Let \( r \) be the \((n,s)\)-type
series corresponding to
\( (\fg[\![x]\!], \delta, \varphi) \), then, under the identification
\( \fg[\![x]\!] \cong \Delta \),
we have the following identities:
\begin{itemize}
    \item \( [a \ot 1 + 1 \ot a, r] = -\delta(a) \)
    for any \( a \in \fg[\![x]\!] \) and
    \item \( \textnormal{CYB}(r) = -\varphi \).
\end{itemize}
The same is true if \( r \) is interpreted as
an element in \( (\fg \ot \fg)(\!(x)\!)[\![y]\!] \).
\end{mainproposition}

In view of this result we call skew-symmetric
\( (n,s) \)-type series quasi-\(r\)-matrices.

Repeating the ideas from
\cite{Drinfeld_quasi_hops_algebras} and
\cite{alekseev_kosmann} we show
that topological quasi-Lie bialgebras
can be twisted similar to topological Lie bialgebras.
More precisely, if \( \delta \)
is a quasi-Lie bialgebra structure
on \( \fg[\![x]\!] \), given by
the Lagrangian subspace \( W \),
and \( s = \sum_{i} a_i \ot b^i \in (\fg \ot \fg)[\![x,y]\!] \) 
is an arbitrary skew-symmetric tensor, then
\begin{equation}
    W_s \coloneqq \left\{ \sum_{i} B(b^i, w)a_i - w \mid w \in W  \right\}
\end{equation}
is another (twisted) Lagrangian subspace
complementary to the diagonal.
This observation implies, that in order to
classify all topological
quasi-Lie bialgebra structures on
\( \fg[\![x]\!] \) up to twisting
it is enough to find one single Lagrangian
subspace within each \( L(n,\alpha) \)
and \( L(\infty) \).
Moreover, allowing substitutions of the form
\( x \mapsto x + a_2 x^2 + a_3 x^3 + \dots \),
\( a_i \in F \), we can without loss of generality
assume that our sequence \( \alpha \) has the
form \[ \alpha = (\dots, 0, \alpha_0, 0, \dots, 0). \]
Lagrangian subspaces for such \( L(n,\alpha) \)
and \( L(\infty) \) are constructed in
\cref{sec:Lagrangian_subspaces}.

Using \cref{mainthm:thmA} and \cref{mainprop:B}
we explain how twisting of a Lagrangian subspace
\( W \subset L(n,\alpha) \) is seen at the level
of \(\delta\) and the corresponding
quasi-\(r\)-matrix \(r\).

\begin{maincorollary}
Let \( (\fg[\![x]\!], \delta, \varphi ) \)
be a topological quasi-Lie bialgebra
structure
corresponding to the quasi-\(r\)-matrix \(r\).
If we twist \( W(r) \) with a skew-symmetric tensor
\(s\) we obtain
another topological
quasi-Lie bialgebra
\( (\fg[\![x]\!], \delta_s, \varphi_s) \), such that
\begin{enumerate}
    \item \( W(r)_s = W(r-s) \);
    \item \( \delta_s = \delta + ds \);
    \item \(\varphi_s = \varphi + \textnormal{CYB}(s) - \frac{1}{2}\textnormal{Alt}((\delta \ot 1)s) \).
\end{enumerate}
\end{maincorollary}

Therefore,
to describe all quasi-\(r\)-matrices
up to twisting it is enough
to find one single quasi-\(r\)-matrix
for each \( L(n,\alpha) \).
We achieve that goal in
\cref{sec:corresponding_quasi_r_matr}
by writing out explicitly series
of type \( (n,s) \) for subspaces from
\cref{sec:Lagrangian_subspaces}.

The results above, in particular, show that if 
\( r \) is a quasi-\(r\)-matrix
and \( \delta(a) \coloneqq [a \ot 1 + 1 \ot a, r] \), then the condition
\begin{equation}
    \textnormal{Alt}((\delta \ot 1 \ot 1) \textnormal{CYB}(r)) = 0
\end{equation}
is trivially satisfied.

We conclude the paper by
using \cref{mainthm:thmA}
for construction of Lie algebra
splittings
\( \Delta \add W = L(n,\alpha) \)
and the corresponding
\( (n,s) \)-type series,
which we call generalized \(r\)-matrices.
These constructions are important
in the theory of integrable systems
because of their use in
the Adler-Konstant-Symes (AKS) scheme and the so-called \(r\)-matrix approach; see \cite{adler_moerbeke_vanhaecke,babelon_bernard_talon}.
The subalgebra splittings of 
\( L(0,0) \)
as well as their physical applications were considered in e.g.\
\cite{skrypnyk_infinite_dimensional_Lie_algebras,Skrypnik_phys_ref}.

Our first result tells us that
in order to obtain new generalized
\( r \)-matrices from
subalgebra splittings
\( L(n,\alpha) = \Delta \add W \) with \( n > 2 \),
the subalgebra \( W \) must be
unbounded. Otherwise
the situation can be reduced
to the splitting of \( L(2,\alpha) \).

\begin{mainproposition}
    Let \(L(n,\alpha) = \Delta \dotplus W\) 
    for some subalgebra \(W \subset L(n,\alpha)\).
    Assume \( W \) is bounded, i.e.\ 
    there is an integer \( N > 0 \)
    such that
    \begin{equation*}
        x^{-N}\fg[x^{-1}] \subseteq W_+ \subseteq x^N\fg[x^{-1}],
    \end{equation*}
    where \(W_+\) is the projection of 
    \(W \subset L(n,\alpha) = \fg(\!(x)\!) \oplus \fg[x]/x^n\fg[x]\) 
    on the first component \(\fg(\!(x)\!)\). 
    Then we have the inclusion
    \[\{0\} \times [x^2]\fg[x]/x^n\fg[x] \subseteq W\] and the image \(\widetilde{W}\) under the canonical projection \(L(n,\alpha) \to L(2,\alpha)\) is a subalgebra satisfying \(L(2,\alpha) = \Delta \dotplus \widetilde{W}\). 

\end{mainproposition}


Despite this result we think that
bounded subalgebras
\( W \subset L(n,\alpha) \)
complementary to \( \Delta \)
are still interesting, because
in the case \( \alpha \neq 0 \)
they lead to unbounded
orthogonal complements
\( W^\perp \) which are also
important in view of the AKS scheme.
We give examples of subalgebras
of \(L(n,\alpha) \) with unbounded
orthogonal complements.

\section*{Acknowledgment} The work of R.A. is supported by the DFG project AB-940/1-1.

\section{Topological Manin pairs}%
\label{sec:top_manin_pairs}
Let \( F \) be an algebraically closed field
of characteristic \( 0 \),
\( \fg \) be a finite-dimensional simple \( F \)-Lie algebra
and
\( \fg[\![x]\!] \coloneqq \fg \ot F[\![x]\!] \) be the Lie
algebra with the bracket defined by
\begin{equation*}
  [a \ot f, b \ot g] \coloneqq [a,b] \ot fg,
\end{equation*}
for all \( a,b \in \fg \) and \( f,g \in F[\![x]\!] \).
From now on, we always endow \( F \) with the discrete topology
and view \( \fg[\![x]\!] \) as 
a topological Lie algebra with the \( (x) \)-adic topology.

A \emph{topological Manin pair} is a pair \( (L , \fg[\![x]\!]) \),
where \( L \) is a Lie algebra equipped with an invariant
non-degenerate symmetric bilinear form \( B \), 
such that
\begin{enumerate}
    \item \( \fg[\![x]\!] \subseteq L \) is a Lagrangian 
    Lie subalgebra with respect to \( B \);
    \item for any continuous functional \( T \colon \fg[\![x]\!] \to F \)
    there exists an element \( f \in L \) such that
    \( T = B(f, -) \).
\end{enumerate}
The statements of
\cite[Proposition 2.9]{Stolin_Zelmanov_Montaner} and
\cite[Proposition 3.12]{AMSZ} 
give a description of all
topological Manin pairs.
For precise formulation we need to
repeat the definitions of some specific
Lie algebras with forms
from \cite[Section 3.2]{AMSZ} 
and \cite[Section 2]{Stolin_Zelmanov_Montaner}.
\begin{definition}\label{def:trivial_double}
We define the Lie algebra
\( L(\infty) \coloneqq \fg \ot A(\infty) \), where
\( A(\infty)\) is the unital commutative algebra with underlying space \(\sum_{i \ge 0} Fa_i \dotplus F[\![x]\!] \) and multiplication given by
\begin{align*}
  a_i a_j \coloneqq 0, \ a_ix^j \coloneqq a_{i-j} \text{ for } i \ge j \text{ and } a_i x^j \coloneqq 0 \text{ otherwise}.
\end{align*}
Let \( \tOp \colon A \to F \) 
be the functional, given
by \( \tOp(a_0) \coloneqq 1 \), \( \tOp(a_i) \coloneqq 0 \), \( i \ge 1 \) and
\( \tOp(F[\![x]\!]) \coloneqq 0 \).
We equip \( L(\infty) \)
with the symmetric non-degenerate
invariant bilinear form
\begin{equation}
    B\left(a \ot \left(\sum_{i \ge 0} c_i a_i, f(x)\right), b \ot \left(\sum_{i \ge 0} t_i a_i, g(x)\right) \right) \coloneqq
    \kf(a,b)\tOp\left(g(x)\sum_{i \ge 0} c_i a_i +  f(x)\sum_{i \ge 0} t_i a_i\right).
\end{equation}
\end{definition}
\begin{definition}
Let \( n \ge 1 \) and \( \alpha = (\alpha_i \in F \mid -\infty < i \le n-2) \) be an arbitrary sequence.
Consider the algebra
\begin{align*}
  A(n, \alpha) \coloneqq F(\!(x)\!) \oplus F[x]/(x^n).
\end{align*}
Abusing the notation we denote the element
\( x^{-n} + \alpha_{n-2}x^{-n+1} + \dots + \alpha_0 x^{-1} + \dots \in F(\!(x)\!) \)
with the same letter \( \alpha \).
Define the functional
\( \tOp \colon A(n, \alpha) \to F\)
by 
\begin{align*}
\tOp(f, [p]) \coloneqq \textnormal{res}_0 \left\{ \alpha(f- p) \right\}.
\end{align*}
Taking the tensor product
of \( A(n, \alpha) \) with \( \fg \)
we get the Lie algebra
\( L(n, \alpha) \coloneqq\fg \ot A(n, \alpha) \),
which we equip with the form
\begin{equation}%
\label{eq:form_L_n_alpha}
  B(a \ot (f, [p]), b \ot (g,[q]))
  \coloneqq \kappa(a,b)\tOp(fg,[pq]).
\end{equation}
It is known that the bilinear form
\( B \) is symmetric non-degenerate and
invariant. 
\end{definition}

\begin{definition}
Take an arbitrary sequence
\( \alpha = (\alpha_i \in F \mid -\infty < i \le -2) \) and let
\( A(0, \alpha) \coloneqq F(\!(x)\!) \).
We define the functional
\( \tOp \colon A(0, \alpha) \to F\) by
\begin{align*}
\tOp(f) \coloneqq \textnormal{res}_0 \left\{ \alpha f \right\},
\end{align*}
where \( \alpha = 1 + \alpha_{-2}x + \dots \in F(\!(x)\!) \). 
We equip the Lie
algebra \( L(0, \alpha) \coloneqq\fg \ot A(0, \alpha) \)
with the bilinear form
\begin{equation}
  B(a \ot f, b \ot g)
  \coloneqq \kappa(a,b)\tOp(fg),
\end{equation}
which is again symmetric
non-degenerate and invariant.
From now on we identify
\( F(\!(x)\!) \) with
\( F(\!(x)\!) \times \{ 0 \} \)
and write \( (f,0) \) for elements
in \( A(0, \alpha) \).
\end{definition}

\begin{definition}
A series of the form
\( \varphi = x + a_2 x^2 + a_3 x^3 + \dots \in F[\![x]\!]  \)
is called a \emph{coordinate transformation}.
Coordinate transformations 
form a group \( \textnormal{Aut}_0 F[\![x]\!] \) under substitution which we view
as a subgroup of automorphisms of \(F[\![x]\!]\).

An element \( \varphi \in \textnormal{Aut}_0 F[\![x]\!]\)
induces an automorphism of
\( A(n, \alpha) \) by
\( f/g \mapsto \varphi(f)/\varphi(g) \)
and \( [p] \mapsto [\varphi(p)] \)
that changes the functional
\( \tOp \) to \( \tOp \circ \, \varphi \).
We write \( A(n,\alpha)^{(\varphi)} \)
for the algebra \( A(n,\alpha) \)
with the functional \( \tOp \circ \, \varphi \).
It is not hard to see that for any \( \varphi \in \textnormal{Aut}_0 F[\![x]\!] \) there is a sequence
\( \beta \) such that
\( A(n,\alpha)^{(\varphi)} = A(n, \beta) \).
\end{definition}

Let \( (L, \fg[\![x]\!]) \) be a topological
Manin pair.
According to
\cite[Proposition 2.9]{Stolin_Zelmanov_Montaner}
as a Lie algebra with form
\( L \cong L(\infty) \) or
\( L \cong L(n, \alpha)\),
for some \( n \ge 0 \) and
some sequence \( \alpha \).
Here we identify \( \fg[\![x]\!] \)
with the diagonal
\[ \Delta \coloneqq \{ (f, [f]) \mid f \in \fg[\![x]\!] \}
\subset L(n, \alpha). \]
Moreover, we can assume
that all the elements \( \alpha_i \) 
in the sequence
\( \alpha \),
except maybe \( \alpha_0 \),
are \( 0 \) 
by virtue of the following result.

\begin{proposition}
{\cite[Proposition 3.12]{AMSZ}}%
\label{prop:Manin_pairs}
Let \(n \ge 0\) and \( \alpha = (\alpha_i \in F \mid -\infty < i \le n-2) \) be a sequence. There exists a  \(\varphi \in \textnormal{Aut}_0F[\![x]\!]\) such that \(A(n,\alpha) \cong A(n,\beta)^{(\varphi)}\), where \(\beta\) is the sequence satisfying \(\beta_i = 0\) for all \( i \neq 0\) and \(\beta_0 = \alpha_0\).
\end{proposition}



\begin{remark}%
\label{rem:diff_equation}
Observe that the result
of \cref{prop:Manin_pairs}
can be interpreted in terms of a formal differential equation. 
Consider an arbitrary \(\alpha(x) = x^{-n} + \alpha_{n-2}x^{-n+1} + \dots + \alpha_0x^{-1} + \dots \in F(\!(x)\!)\) and \(\beta(x) = x^{-n} + \alpha_0x^{-1}\). 
Then the functionals \(\tOp_\alpha\) and \(\tOp_\beta\) defined
on \(A(n,\alpha)\) and
\(A(n,\beta)\) respectively
are given by
\begin{equation*}
    \tOp_{\alpha}(f,[p]) = \textnormal{res}_0\{\alpha(f- p)\} \ \
    \text{ and }
    \ \
    \tOp_{\beta}(f,[p]) = \textnormal{res}_0\{\beta(f- p)\})
\end{equation*}
The equality
\(A(n,\alpha)^{(\varphi)} = A(n,\beta)\) can be
expressed as
\begin{equation}
    \textnormal{res}_0\{\beta(x)f(x)\} = \textnormal{res}_0\{\alpha(x)f(\varphi(x))\} = \textnormal{res}_0\{\alpha(\psi(x))f(x)\psi'(x)\}, 
\end{equation}
where \(\psi \in \textnormal{Aut}_0(F[\![x]\!])\) is the compositional inverse of \(\varphi\), i.e.\ \(\varphi(\psi(x)) = x\). Since the residue pairing is non-degenerate on \(F(\!(x)\!)\), we obtain
\begin{equation}%
\label{eq:diff_equation}
    \alpha(\psi(x))\psi'(x) = \beta(x).
\end{equation}
In particular, the transformation \(\varphi\) is the compositional inverse 
of the solution to
\cref{eq:diff_equation}.
\end{remark}

\section{Series of type \((n,s)\) and subspaces of \(L(n,\alpha)\)}
Let \(\{b_i\}_{i = 1}^d\) be
an othonormal basis of \( \fg \) with
respect to the Killing form \( \kf \).
We write \( \Omega \)
for the quadratic Casimir element
\( \sum_{i = 1}^d b_i \otimes b_i \in \fg \ot \fg\). 
It satisfies the identity
\([a \ot 1 + 1 \ot a, \Omega] = 0 \)
for all \(a \in \fg\). 

In this section 
we describe a bijection between
subspaces
\( W \subset L(n, \alpha) \)
complementary to 
\( \Delta \)
and certain 
series.
The following definition
introduces convenient spaces
containing these series.
\begin{definition}
We put \( A_1(n,\alpha) \coloneqq A(n, \alpha) = F(\!(x_1)\!) \oplus F[x_1]/(x_1^n) \)
and then define inductively the algebras
\begin{equation}
    A_m(n, \alpha) \coloneqq A_{m-1}(n, \alpha)(\!(x_m)\!) \oplus A_{m-1}(n, \alpha)[x_m]/x_m^nA_{m-1}(n, \alpha), \ m > 1.
\end{equation}
The functional \( \tOp \) defined on 
\( A(n, \alpha) \) extends
inductively to a functional on
\( A_m(n, \alpha) \). More precisely, 
\begin{equation}%
\label{eq:trace_extension}
    \tOp\left( \sum_{k \ge -N} f_k x_m^k, \sum_{\ell = 0}^{n-1} [g_\ell x_m^\ell] \right) \coloneqq \sum_{k \ge -N} \tOp (f_k)\tOp(x_m^k, 0) + \sum_{\ell = 0}^{n-1}\tOp(g_\ell)\tOp(0, [x_m]^\ell),
\end{equation}
where \( f_k,g_\ell \in A_{m-1}(n,\alpha) \).
Since \( \tOp(x^n F[\![x]\!]) = 0 \),
the sum on the right-hand side of
\cref{eq:trace_extension} is finite and
well-defined. This allows us to extend the
form \( B \) on \( L(n, \alpha) \) to
a symmetric non-degenerate bilinear form
on the \( \fg \)-module
\begin{equation}
    L_m(n, \alpha) \coloneqq \fg^{\ot m} \ot A_m(n, \alpha)
\end{equation}
by letting
\begin{equation}
    B((a_1 \ot \dots \ot a_m)\ot f,(b_1 \ot \dots \ot b_m)\ot g) \coloneqq
    \tOp(fg) \prod_{k=1}^m \kf(a_k, b_k),
\end{equation}
for all \(a_1,\dots,a_m,b_1,\dots,b_m \in \fg\) and \(f,g \in A_{m}(n,\alpha)\).
\end{definition}

Fix some integer
\(n \ge 0\). 
We interpret the quotient \( y^n \Omega / (x-y) \) in the following way
\begin{equation}%
\label{eq:def_wki}
\begin{aligned}
    \frac{y^n\Omega}{x-y} &= \sum_{k = 0}^{n-1}\sum_{i = 1}^d b_i(0,-[x]^{(n-1)-k}) \otimes b_i(y^k, [y]^k)  + \sum_{k = n}^\infty \sum_{i = 1}^d b_i(x^{(n-1)-k},0) \otimes b_i(y^k, 0) \\
    &= \sum_{k = 0}^{\infty}\sum_{i = 1}^d w_{k,i} \otimes b_i(y^k,[y]^k) \in \left(L(n, \alpha) \ot \fg \right)[\![(y, [y])]\!] \subset L_2(n, \alpha),
\end{aligned}
\end{equation}
where \( \alpha \) is an
arbitrary sequence and we write
\( b_i(x^\ell, [x]^m) \) meaning 
\( b_i \ot (x^\ell, [x]^m) \).

\begin{definition}%
\label{def:n_s_series}
Since
\( \left(L(n, \alpha) \ot \fg \right)[\![(y,[y])]\!] \) is an \( F[\![x]\!] \cong F[\![(x, [x])]\!] \)-module and 
\[
(\fg \ot \fg)[\![x,y]\!] \cong (\Delta \ot \fg)[\![(y, [y])]\!]  \subset \left(L(n, \alpha) \ot \fg \right)[\![(y,[y])]\!] 
\]
the series
\begin{equation}%
\label{eq:standard_form}
    r(x,y) = \frac{s(x) y^n \Omega}{x-y} + g(x,y),
\end{equation}
where \( g \in (\fg \ot \fg)[\![x,y]\!] \) and
\( s \in F[\![x]\!]^\times \),
is also inside \( \left(L(n, \alpha) \ot \fg \right)[\![(y,[y])]\!] \).
Series of the form \cref{eq:standard_form}
are called \emph{series of
type \((n,s)\)}.
\end{definition}

\begin{remark}%
\label{rem:represetnations_of_any_series}
Every series
\[
r(x,y) = \frac{h(x,y)\Omega}{x-y} + g(x,y) \in L_2(n, \alpha),
\]  
where \(h \in F[\![x,y]\!]\), \(h(x,x) \neq 0\) and 
\(g \in (\fg\otimes \fg)[\![x,y]\!]\),
has a unique representation as a series of type \((n,s)\). 
Indeed, write \(h(x,x) = x^n s(x)\) for some
\( s \in F[\![x]\!]^\times \).
Then \(h(x,y) - y^n s(x) = (x-y)f(x,y)\) for some \(f \in F[\![x,y]\!]\).
This implies that we can rewrite \( r \) in the \( (n, s) \) form
\begin{equation}
\begin{aligned}
    r(x,y) &= \frac{s(x) y^n\Omega}{x-y} + f(x,y)\Omega + g(x,y).
\end{aligned}
\end{equation}
In the construction of \(f\) we are 
using the fact that
for any \(F\)-vector space \(V\) 
and any element \( h \in V[\![x,y]\!] \)
\begin{equation}\label{eq:series_vanishing_at_diagonal}
    h(z,z) = 0 \implies h(x,y) = (x-y)f(x,y)
\end{equation}
for some  \( f\in V[\![x,y]\!] \).
\end{remark}


\begin{definition}
For each series \( r \)
of type \( (n, s) \)
we define another series
\(\overline{r}\) of the same type
\( (n,s) \) by
\begin{equation}%
\label{eq:tau_r}
    \overline{r}(x,y) \coloneqq 
    \frac{s(y)x^n \Omega}{x-y} - \tau(g(y,x))
    \in \left( L(n,\alpha)\ot \fg \right)[\![(y,[y])]\!],
\end{equation}
where \( \tau \) is the
\( F[\![x,y]\!] \)-linear extension
of the map \( a \ot b \mapsto b \ot a \).
To see that this is an \( (n,s) \)-type
series its enough to apply the argument
from \cref{rem:represetnations_of_any_series}.
Series of type \( (n,s) \),
satisfying \( r = \overline{r} \),
are called \emph{skew-symmetric}.
\end{definition}

\begin{definition}\label{def:gcybe}
\emph{The generalized classical Yang-Baxter
equation (GCYBE)} is the equation
for an \( (n,s) \)-type series of
the form
\begin{equation}
        \textnormal{GCYB}(r) \coloneqq [r^{12}(x_1,x_2),r^{13}(x_1,x_3)] + [r^{12}(x_1,x_2),r^{23}(x_2,x_3)] + [r^{13}(x_1,x_3),\overline{r}^{23}(x_2,x_3)] = 0.
    \end{equation}
Here \( (-)^{13} \colon L_2(n, \alpha) \to 
(U(\fg) \ot U(\fg) \ot U(\fg))\ot A_3(n, \alpha) \) is the inclusion map
given by
\[ a \ot b \ot \left( \sum_{k \ge -N} F(x_1,[x_1])x_2^k, \sum_{m = 0}^{n-1} G(x_1,[x_1])[x_2]^m \right) \mapsto
a \ot 1 \ot b \ot \left(\sum_{k \ge -N} F(x_1,[x_1])x_3^k, \sum_{m = 0}^{n-1} G(x_1,[x_1])[x_3]^m \right).
\]
Other inclusions are defined
in a similar manner.
The commutators are then taken in
the associative \( A_3(n,\alpha) \)-
algebra \( (U(\fg) \ot U(\fg) \ot U(\fg))\ot A_3(n, \alpha) \).
\end{definition}

Before formulating the main theorem
of the section we note that if \( \alpha = (\alpha_i \in F \mid -\infty < i \le n-2) \)
is an arbitrary sequence and
\(\alpha(x) = x^{-n} + \alpha_{n-2}x^{-n+1} + \dots + \alpha_0x^{-1} + \dots \in F(\!(x)\!)\) is the corresponding series,
then \( x^n \alpha(x) \in F[\![x]\!]^\times \).

\begin{theorem}%
\label{thm:series_and_subspaces}
Let \(n \in \mathbb{Z}_{\ge 0} \) 
and \( \alpha = (\alpha_i \in F \mid -\infty < i \le n-2) \) be an
arbitrary sequence with the corresponding series \( \alpha(x) \in F(\!(x)\!) \).
Consider the map 
\[
W \colon L_2(n,\alpha) \longrightarrow \{V \subset L(n,\alpha)\mid V \textnormal{ is a subspace}\}
\] given by
\begin{equation*}
    \sum_{i, j} b_i \ot b_j \ot \left( \sum_{k \ge -N_i} (f^{ij}_k, [p^{ij}_k])x^k, \sum_{m=0}^{n-1} (g^{ij}_m, [q^{ij}_m]) [x]^m \right) \mapsto
    \textnormal{span}_F \left\{ b_i(f^{ij}_k, [p^{ij}_k]) \mid k \ge -N, 1 \le i, j \le d \right\}.
\end{equation*}
The following results are true:
\begin{enumerate}
    \item \(W\) defines a bijection between series of type \(\left(n,\frac{1}{x^n \alpha(x)}\right)\) and subspaces \(V \subseteq L(n,\alpha)\)
    complementary to the diagonal
    \( \Delta \), i.e.\
    \( L(n,\alpha) = \Delta \add V;\)

    \item For any series \( r \) of type \(\left(n,\frac{1}{x^n \alpha(x)}\right)\)
    we have \( W(r)^\perp = W(\overline{r}) \) inside \( L(n, \alpha) \);
    
    \item Any series \(r\) of type \(\left(n,\frac{1}{x^n \alpha(x)}\right)\) satisfies \(\textnormal{GCYB}(r) = \psi\),
    where \(\psi \in (\fg \otimes \fg \otimes \fg)[\![(x_1, [x_1]),(x_2, [x_2]),(x_3, [x_3])]\!]\) is defined by 
    \[B(v_1 \otimes v_2 \otimes v_3,\psi) = B(v_1,[v_2,v_3])\]
    for all \(v_1 \in W(\overline{r}),v_2,v_3 \in W(r)\).
\end{enumerate}
\end{theorem}
\begin{proof}
Fix an \( \left(n,\frac{1}{x^n \alpha(x)}\right) \)-type
series
\begin{equation*}
\begin{aligned}
    r(x,y) &= \frac{1}{x^n \alpha(x)} \frac{y^n \Omega}{x-y} + g(x,y) \\
    &= \sum_{k = 0}^{\infty} \sum_{i = 1}^d s_{k,i} \ot b_i(y^k,[y]^k) + \sum_{k = 0}^{\infty} \sum_{i = 1}^d
    g_{k,i} \ot b_i(y^k,[y]^k) \in (L(n,\alpha) \ot \fg)[\![(y,[y])]\!].
\end{aligned}
\end{equation*}
It is easy to see that
\[
U \coloneqq \textnormal{span}_{F} \{ w_{k,i} \mid k \ge 0, 1 \le k \le d \} \subset L(n, \alpha),
\]
where \( w_{k,i} \) are defined in
\cref{eq:def_wki}, 
satisfies the condition
\( \Delta \add U = L(n, \alpha) \).
Since \( s \coloneqq \frac{1}{x^n \alpha(x)} \) is invertible, we have
\( sU \add s\Delta = sU \add \Delta = L(n, \alpha) \). In other words, the
space
\begin{equation}
    sU = \textnormal{span}_{F}\{ s_{k,i} = sw_{k,i} \mid k \ge 0, 1 \le k \le d \} \subset L(n, \alpha)
\end{equation}
is also complementary to the diagonal.
Finally, since \( g_{k,i} \in \Delta \) 
the space
\[
W(r) = \textnormal{span}_F\{ sw_{k,i} + g_{k,i} \mid k \ge 0, 1 \le k \le d \} \subset L(n, \alpha)
\]
is complementary to the diagonal.
Conversely, if \( V \subset L(n, \alpha) \) satisfies \( V \add \Delta = L(n, \alpha) \), 
then for each
\( k\ge 0 \) and \( 1 \le i \le d \) we can find a
unique \( g_{k,i} \in \Delta \)
such that \( s w_{k,i} + g_{k,i} \in V \).
Define the \( (n,s) \) series 
\( r_V \) by
\[
r_V(x,y) = \sum_{k \ge 0} \sum_{i = 1}^d
(s w_{k,i} + g_{k,i}) \ot b_i (y^k,[y]^k).
\]
It is now clear, that
\( W(r_V) = V \). These constructions 
establish the bijection in part 1.

To prove the second statement, observe
that
\begin{equation}%
\label{eq:dual_basis_in_series}
B(s w_{k,i},b_j (y^\ell,[y]^\ell)) = \delta_{i,j}\delta_{k,\ell}.
\end{equation}
Furthermore, the straightforward calculation shows that
\begin{equation*}
\begin{aligned}
B(s w_{k,i},s w_{\ell,j}) &= \begin{cases}
- \textnormal{res}_0 \left\{ s x^{(n-1) - k-\ell - 1} \right\} & \textnormal{ if } \, i = j \, \textnormal{ and } \, 0 \le k,\ell \le n-1, \\
\textnormal{res}_0 \left\{ s x^{(n-1) - k-\ell - 1} \right\} & \textnormal{ if } \, i = j \, \textnormal{ and } \, k,\ell \ge n, \\
0 & \textnormal{ otherwise},
\end{cases} \\
&=
\begin{cases}
- s_{k + \ell - n + 1} & \textnormal{ if } \, i = j, \, 0 \le k,\ell \le n-1 \, \textnormal{ and } \, k + \ell \ge n-1, \\
s_{k + \ell - n + 1} & \textnormal{ if } \, i = j \, \textnormal{ and }k \, ,\ell \ge n, \\
0 & \textnormal{ otherwise},
\end{cases}
\end{aligned}
\end{equation*}
where \(s(x) = \sum_{k = 0}^\infty s_k x^k\). 
We write
\begin{equation*}
\begin{aligned}
    \overline{r}(x,y) &= 
    \frac{s(y)x^n \Omega}{x-y} - \tau(g(y,x))
    =
    \frac{s(x)y^n \Omega}{x-y} -
    \frac{(s(x)y^n - s(y)x^n)\Omega}{x-y}
    -\tau(g(y,x)) \\
    &=
    \sum_{k \ge 0} \sum_{i = 1}^d
    (s w_{k,i} + \overline{g}_{k,i}) \ot b_i (y^k,[y]^k).
\end{aligned}
\end{equation*}
Consider the quotient
\begin{equation*}
\begin{aligned}
    &\frac{(s(x)y^n-s(y)x^n)\Omega}{x-y} = \frac{y^n(s(x) - s(y))\Omega}{x-y} - \frac{s(y)(x^n - y^n) \Omega}{x-y}\\
    &=
    \sum_{k \ge 0} \sum_{i = 1}^d s_k \left( \sum_{\ell = 1}^k  b_i (x^{k-\ell}, [x]^{k - \ell}) \ot b_i (y^{(n-1)+\ell}, [y]^{(n-1)+\ell})
    - 
    \sum_{\ell = 1}^n
    b_i (x^{n-\ell}, [x]^{n-\ell})
    \ot b_i(y^{k+\ell - 1}, [y]^{k+\ell - 1}) \right).
\end{aligned}
\end{equation*}
The coefficient of \( b_i(x^{k}, [x]^k) \ot b_i (y^{\ell}, [y]^{\ell}) \)
in the expression above is
\begin{equation*}
\begin{aligned}
    -s_{k+\ell-(n-1)} & \ \text{ if } \, 0 \le k, \ell \le n-1 \, \text{ and } \, k + \ell \ge n-1, \\
    s_{k+\ell-(n-1)} & \ \text{ if } \,
    k, \ell \ge n,
\end{aligned}
\end{equation*}
which coincides with \(B(s w_{k,i}, s w_{\ell, i})\).
If we now expand
the coefficients \( g_{k,i} \) in the
following way
\begin{equation*}
    g_{k,i} = \sum_{\ell \ge 0} \sum_{j = 1}^d g_{k,i}^{\ell,j} b_j (x^\ell,[x]^\ell),
\end{equation*}
the coefficients \( \overline{g}_{k,i}\)
can be rewritten as
\begin{equation*}
    \overline{g}_{k,i} = 
    - \sum_{\ell \ge 0} \sum_{j = 1}^d
    (g^{k,i}_{\ell, j} + B(s w_{k,i}, s w_{\ell, j})) b_i (x^k, [x]^k) \ot b_j (y^\ell, [y]^\ell).
\end{equation*}
Combining all the results
above we obtain
the desired equality
\begin{equation*}
\begin{aligned}
    B(s w_{k,i} + g_{k,i},s w_{\ell,j} + \overline{g}_{\ell,j}) &=
    B(sw_{k,i}, sw_{\ell, j})
    + B(sw_{k,i},\overline{g}_{\ell,j})
    +B(g_{k,i}, sw_{\ell, j})
    +B(g_{k,i},\overline{g}_{\ell,j}) \\
    &=
    B(sw_{k,i}, sw_{\ell, j}) + 
    (-g_{k,i}^{\ell,j} -  B(sw_{k,i}, sw_{\ell, j})) + g_{k,i}^{\ell,j} + 0 \\
    &= 0
\end{aligned}
\end{equation*}
which completes the proof of the second statement.

Using the same technique
as in \cite[Section 1]{abedin2021geometrization},
one can prove that
\begin{equation*}
    \psi \coloneqq \textnormal{GCYB}(r) \in
(\Delta \ot \fg \ot \fg)[\![(x_2,[x_2]), (x_3,[x_3])]\!]
\end{equation*}
for any series \(r\) of type \( (n,s) \).
Define \( r_{k,i} \coloneqq s w_{k,i} + g_{k,i} \) and
\( \overline{r}_{k,i} \coloneqq s w_{k,i} + \overline{g}_{k,i} \)
and rewrite 
\(\textnormal{GCYB}(r)\) as
\begin{equation}%
\label{eq:gcybe_subalgebra}
\begin{aligned}
\psi =& \sum_{k, \ell \ge 0} \sum_{i,j = 1}^d
[r_{k,i}, r_{\ell, j}] \ot b_i (x_2^k,[x_2]^k) \ot b_j (x_3^\ell, [x_3]^\ell) \\ 
&+ \sum_{k \ge 0} \sum_{i = 1}^d
r_{k,i} \ot \left([b_i (x_2^k, [x_2]^k) \ot (1,1), r(x_2, x_3)] + [(1,1) \ot b_i(x_3^k, [x_3]^k),
 \overline{r}(x_2,x_3)]\right).
\end{aligned}
\end{equation}
Applying 
\(B(\overline{r}_{k_1,i_1} \otimes r_{k_2,i_2} \otimes r_{k_3,i_3},-)\) to the equation above, we get
\begin{equation}
    B(\overline{r}_{k_1,i_1} \otimes r_{k_2,i_2} \otimes r_{k_3,i_3}, \psi) = B(\overline{r}_{k_1,i_1},[r_{k_2,i_2},r_{k_3,i_3}]).
\end{equation}
This gives the last statement
because \( W(r) \) and \( W(\overline{r})\)
are generated by \( r_{k,i}\) and 
\( \overline{r}_{k,i}\) respectively.
\end{proof}

\begin{corollary}%
\label{cor:series_and_subspaces}
    Let \( n \in \mathbb{Z}_{\ge 0} \), \( \alpha = (\alpha_i \in F \mid -\infty < i \le n-2) \)  and \( W \)
    be as in \cref{thm:series_and_subspaces}.
    Then
    \begin{enumerate}
        \item \( W \)
    defines a bijection between skew-symmtric
    \(\left(n,\frac{1}{x^n \alpha(x)}\right)\)-type series
    and Lagrangian subspaces \(V \subseteq L(n,\alpha)\) 
    complementary to the diagonal \( \Delta \);
    \item \(W\) defines a bijection between \(\left(n,\frac{1}{x^n \alpha(x)}\right)\)-type series solving GCYBE and subalgebras \(V \subseteq L(n,\alpha)\) complementary to the diagonal \( \Delta \).
    \end{enumerate}
\end{corollary}

As we can see from the proof of
\cref{thm:series_and_subspaces}
the element \( \psi \) 
in
\( \textnormal{GCYB}(r) = \psi  \)
represents the
obstruction for \( W(r) \) 
from being
a Lie subalgebra. 
This observation raises an interesting question that
we do not consider in this paper: what elements \( \psi \) can
appear on the right-hand side of the above-mentioned equation.

Observe that if
\( r \) is a series of type
\( (n,s) \) and it satisfies
\begin{equation}
\label{eq:cybe}
        \textnormal{CYB}(r) \coloneqq [r^{12}(x_1,x_2),r^{13}(x_1,x_3)] + [r^{12}(x_1,x_2),r^{23}(x_2,x_3)] + [r^{13}(x_1,x_3),{r}^{23}(x_2,x_3)] = \psi
\end{equation}
for some 
\( \psi \in (\fg \ot \fg \ot \fg)[\![x,y,z]\!]\),
then \( r \) is automatically
skew-symmetric and hence
solves the first equation as well.
To prove that one can e.g.\ 
repeat
the argument from
\cite[Lemma 5.2]{AMSZ}.
In other words, for a fixed
\( \psi \) solutions to 
\(\textnormal{CYB}(r) = \psi\) 
form a subclass of solutions 
to
\(\textnormal{GCYB}(r) = \psi\).
In particular, solutions to 
\(\textnormal{CYB}(r) = 0\).
are exactly the skew-symmetric solutions
to \(\textnormal{GCYB}(r) = 0\) .
We call the
equation 
\(\textnormal{CYB}(r) = \psi\)
\emph{Manin-Yang-Baxter equation}.

\begin{remark}%
\label{rem:projections_vs_large_r}
As our notation suggest, we could
have interpreted
\(y^n\Omega / (x-y) \) as
\begin{equation*}
    \frac{y^n \Omega}{x-y} = \sum_{k \ge 0} \sum_{i=1}^d b_i x^{-k-1} \ot b_i y^{n + k} \in (\fg \ot \fg)(\!(x)\!)[\![y]\!]
\end{equation*}
and performed all the arithmetic calculations in this form. 
To restore an \( (n,s) \)-type series
from
\begin{equation}%
\label{eq:projection_standard_form}
    \frac{s(x)y^n \Omega}{x-y} + g(x,y) \in (\fg \ot \fg)(\!(x)\!)[\![y]\!] 
\end{equation}
we can simply view \( s(x) \in F[\![x]\!]^\times \) and
\( g(x,y) \in (\fg \ot \fg)[\![x,y]\!] \)
as elements in
\( F[\![(x,[x])]\!]^\times \) and
\( (\fg \ot \fg)[\![(x,[x]),(y,[y])]\!] \)
respectively
and reinterpret the singular part
\(y^n \Omega / (x-y) \) as it was done in
\cref{eq:def_wki}.

Conversely, to get a series of the form
\cref{eq:projection_standard_form}
from a series of type \( (n,s) \)
we can just project the latter
onto the first component.

In other words, we have a bijection
between \( (n,s) \)-type series
in \( L_2(n, \alpha) \)
and their projections \cref{eq:projection_standard_form}
onto the first component given
by different interpretations of
the singular part
\(y^n \Omega / (x-y) \).

Although, all arithmetic operations
can be performed
in the form \cref{eq:projection_standard_form},
the construction of \( W(r) \) and
statements like \( \Delta \cap W(r) = 0 \)
require us to pass to the interpretation
\cref{eq:def_wki}. This is our main
motivation to work directly
with \( (n,s) \)-type series in \( L_2(n,\alpha) \)
instead of their projections.
\end{remark}

In view of
\cref{rem:projections_vs_large_r},
we have a new proof of
\cite[Corollary 5.5]{AMSZ}.

\begin{corollary}
Classical (formal) \(r\)-matrices, i.e.\
skew-symmetric elements
\begin{equation}%
    \frac{s(x)y^n \Omega}{x-y} + g(x,y) = 
    \frac{1}{x^n \alpha(x)}\frac{y^n \Omega}{x-y} + g(x,y) \in (\fg \ot \fg)(\!(x)\!)[\![y]\!],
\end{equation}
solving GCYBE,
are in bijection with skew-symmetric
series of type \( (n,s) \) solving GCYBE
and hence in bijection with
Lagrangian Lie subalgebras of 
\( L(n, \alpha) \) complementary
to the diagonal \( \Delta \).
\end{corollary}


The result
of \cite[Theorem 5.6]{AMSZ}
can be now formulated in the following way.
\begin{corollary}
Skew-symmetric series of type
\( \left(n, \frac{1}{x^n \alpha(x)}\right) \)
that also solve GCYBE exist only for
\( n = 0, 1\) and \( n = 2 \) with \( \alpha_0 = 0 \).
\end{corollary}

 \section{Quasi-Lie bialgebra structures on \(\fg[\![x]\!]\)}
We remind that
\( F \) is a discrete algebraically
closed field of characteristic \( 0 \)
and \( \fg[\![x]\!] \) is
an \( F \)-Lie algebra equipped
with the \( (x) \)-adic topology.

As we now know, series of type
\( \left(n, 1/(x^n \alpha(x))\right) \) solving CYBE \cref{eq:cybe} are in bijection with
Lagrangian subalgebras
\( W \subset L(n,\alpha) \)
complementary to the diagonal.
On the other hand, such Lagrangian
subalgebras are in bijection
with non-degenerate topological Lie bialgebra structures.
See \cite{AMSZ} for their definition 
and classification.

It turns out, that if we drop the condition
on \( W \) being a subalgebra, we
get so called (non-degenerate) topological quasi-Lie
bialgebras. This section is devoted
to their classification up to
topological twists and coordinate
transformations.

\begin{definition}%
\label{def:top_quasi_lie_bialg}
A \emph{topological quasi-Lie bialgebra}
structure on \( \fg[\![x]\!] \)
consists of
\begin{itemize}
    \item a skew-symmetric continuous linear map
    \( \delta \colon \fg[\![x]\!] \to (\fg \ot \fg)[\![x,y]\!] \) and
    
    \item a skew-symmetric element
    \( \varphi \in (\fg \ot \fg \ot \fg)[\![x,y,z]\!] \),
\end{itemize}
which are subject to the following conditions
\begin{enumerate}
    \item \(\delta([a,b]) = [a \ot 1 + 1 \ot a, \delta(b)] - [b \ot 1 + 1 \ot b , \delta (a)]\), i.e.\  \( \delta \) is a \( 1 \)-cocycle;
    
    \item \(\frac{1}{2}\textnormal{Alt}((\delta \ot 1)\delta(a)) = 
    [a \ot 1 \ot 1 + 1 \ot a \ot 1 + 1 \ot 1 \ot a, \varphi]\);
    
    \item \( \textnormal{Alt}((\delta \ot 1 \ot 1)\varphi) = 0 \),
\end{enumerate}
where \( \textnormal{Alt}(x_1 \ot \dots \ot x_n) \coloneqq \sum_{\sigma \in S_n} \textnormal{sgn}(\sigma)x_{\sigma(1)} \ot \dots \ot x_{\sigma(n)} \).
\end{definition}

\begin{lemma}%
\label{lem:M_pair_q_bial_correspondance}
There is a one-to-one correspondence 
between 
triples
\((L, \fg[\![x]\!], W)\),
where \( (L, \fg[\![x]\!]) \) is a topological
Manin pair and \( W \subset L \) is a Lagrangian
subspace satisfying \( W \add \fg[\![x]\!] = L \),
and
quasi-Lie bialgebra structures 
on \( \fg[\![x]\!]\).
\end{lemma}
\begin{proof}
We start with a topological Manin
pair \( (L ,\fg[\![x]\!])\).
If \( W \subset L \) is a Lagrangian subspace
complementary to \(\fg[\![x]\!]\), then it is easy
to see that \( W \cong \fg[\![x]\!]' \).
Therefore, we have an isomorphism of vector spaces
\[ L \cong \fg[\![x]\!] \add \fg[\![x]\!]'. \]
The form on \( L \) under this isomorphism
becomes standard evaluation form \(\langle -,-\rangle\) on
\( \fg[\![x]\!] \add \fg[\![x]\!]' \).
We fix such an isomorphism.

Let us define two linear functions
\begin{equation*}
    p_1 \colon \fg[\![x]\!]' \ot \fg[\![y]\!]' \to \fg[\![x]\!]
    \ \text{ and } \
    p_2 \colon \fg[\![x]\!]' \ot \fg[\![y]\!]' \to \fg[\![x]\!]'
\end{equation*}
by \( [f,g] = p_1(f \ot g) + p_2(f \ot g) \).
We put
\begin{equation*}
\delta \coloneqq p_2^{\vee}
\colon (\fg[\![x]\!]')^\vee \cong \fg[\![x]\!]
\to (\fg[\![x]\!]' \ot \fg[\![y]\!]')^\vee \cong (\fg \ot \fg)[\![x,y]\!],
\end{equation*}
and let \( \psi \in (\fg \ot \fg \ot \fg)[\![x,y,z]\!]  \)
be the unique element satisfying
the condition
\begin{equation}%
\label{eq:condition_psi_p1}
    \langle h, [f,g]\rangle = \langle h, p_1(f \ot g) \rangle
    =
    \langle f \ot g \ot h, \psi \rangle
    \ \text{ for all } f,g,h \in \fg[\![x]\!]'.
\end{equation}
The skew-symmetry of
\( p_2 \) implies the
skew-symmetry of \( \delta \),
whereas the skew-symmetry of
\( p_1 \) and the invariance
of the evaluation form yield the
skew-symmetry of \( \psi \).

Next, we observe that for all
\( a,b \in \fg[\![x]\!] \) and \(f,g \in \fg[\![x]\!]'\) we have
\begin{equation*}
\begin{aligned}
    &\langle [a,f], g \rangle
    = \langle a, [f,g] \rangle
    = \langle a, p_2(f \ot g) \rangle
    = \langle \delta(a), f \ot g \rangle
    = \langle (f \ot 1)\delta(a), g \rangle, \\
    &\langle [a,f], b \rangle
    = - \langle f, [a,b] \rangle
    = - \langle f \circ \ad_a, b \rangle.
\end{aligned}
\end{equation*}
In other words, the invariance of the
form forces the following equality to hold
\begin{equation}%
\label{eq:unique_commutator_on_L}
    [a,f] = - f \circ \ad_a + (f \ot 1)\delta(a).
\end{equation}
Using \cref{eq:unique_commutator_on_L}
and non-degeneracy of the form
we show that \( \delta \) is a 
\( 1 \)-cocycle:
\begin{equation}
\begin{aligned}
  \langle \delta([a,b]), f \ot g  \rangle
  &= \langle [a,b], p_2(f \ot g)  \rangle
  = \langle [a,b], [f , g]  \rangle
  = \langle [[a,b],f], g  \rangle 
  = \langle -[[b,f],a] - [[f,a],b], g  \rangle \\
  &=  \langle [f \circ \ad_b - (f \ot 1)\delta(b),a] - [f \circ \ad_a - (f \ot 1)\delta(a),b], g  \rangle \\
  &= - \langle a, [f \circ \ad_b, g] \rangle
  + \langle b, [f \circ \ad_a, g] \rangle
  + \langle (f \ot \ad_a) \delta(b), g \rangle
  - \langle (f \ot \ad_b) \delta(a), g \rangle \\
  &= \langle [a \ot 1 + 1 \ot a, \delta(b)] - [b \ot 1 + 1 \ot b , \delta (a)], f \ot g \rangle.
\end{aligned}
\end{equation}
The \( 1 \)-cocycle condition implies that \( \delta\)
is continuous as it was noted in \cite[Remark 3.16]{AMSZ}.

For conditions 2 and 3 from the definition
of a topological quasi-Lie bialgebra
consider
the Jacobi identity for \( f,g,h \in \fg[\![x]\!]' \):
\begin{equation}%
\label{eq:Jacobi_identity_in_L}
\begin{aligned}
0 = 
& [p_1(f \ot g), h] + [p_1(g \ot h), f] + [p_1(h \ot f), g] \\
&+ p_1(p_2(f \ot g)\ot h) + p_1(p_2(g \ot h)\ot f) + p_1(p_2(h \ot f)\ot g) \\
&+ p_2(p_2(f \ot g)\ot h) +
p_2(p_2(g \ot h)\ot f) + p_2(p_2(h \ot f)\ot g).
\end{aligned}
\end{equation}
We denote by
\( \circlearrowleft \)
the summation
over circular permutations of 
symbols \( f,g \) and \( h \), e.g.\
\(\circlearrowleft \langle p_1(f \ot g), h \rangle = \langle p_1(f \ot g), h \rangle
+ \langle p_1(g \ot h), f \rangle + 
\langle p_1(h \ot f), g \rangle\).
Applying \( \langle -, a \rangle \)
to \cref{eq:Jacobi_identity_in_L}
for an arbitrary
\( a \in \fg[\![x]\!] \) gives
\begin{equation*}
\begin{aligned}
\langle p_2(p_2 \ot 1)(\circlearrowleft f \ot g\ot h), a \rangle
&= - \langle \circlearrowleft [p_1(f \ot g), h], a \rangle \\
\langle p_2 \ot 1(\circlearrowleft f \ot g \ot h), \delta(a) \rangle
&= \circlearrowleft \langle -h\circ \ad_a, p_1(f \ot g) \rangle \\
\langle \circlearrowleft f \ot g \ot h, (\delta \ot 1)\delta(a) \rangle
&= \circlearrowleft \langle f \ot g \ot (-h\circ \ad_a), \psi \rangle \\
\langle f \ot g \ot h, \textnormal{Alt} ((\delta \ot 1)\delta(a))/2 \rangle
&= - \langle f \ot g \ot h, [1 \ot 1 \ot a + 1 \ot a \ot 1 + a \ot 1 \ot 1,\psi] \rangle,
\end{aligned}
\end{equation*}
where the very last identity
holds because of the
skew-symmetry of \( \psi \).
Multiplying this equality by
\( 2 \) we get the relation
\begin{equation*}
\langle f \ot g \ot h, \textnormal{Alt} ((\delta \ot 1)\delta(a)) + 2 [1 \ot 1 \ot a + 1 \ot a \ot 1 + a \ot 1 \ot 1,\psi] \rangle = 0.
\end{equation*}
Letting \( \varphi \coloneqq - \psi \)
we obtain the second identity from the
definition of a topological
quasi-Lie bialgebra structure.
Applying instead
\( \langle s, - \rangle \),
\( s \in \fg[\![x]\!]'\)
to the Jacobi identity
\cref{eq:Jacobi_identity_in_L}
we get the desired
\begin{equation*}
    \textnormal{Alt}((\delta \ot 1 \ot 1) \psi) = 0.
\end{equation*}
Therefore, \( (\fg[\![x]\!], \delta, \varphi ) \) is a topological
quasi-Lie bialgebra.

For the converse direction,
we put 
\( L \coloneqq \fg[\![x]\!] \add \fg[\![x]\!]' \) with
the standard evaluation form;
we let \( p_1 \) be the unique
element in
\( \textnormal{Hom}_{F-\textnormal{Vect}}(\fg[\![x]\!]' \ot \fg[\![x]\!]', \fg[\![x]\!]) \)
satisfying 
\cref{eq:condition_psi_p1} with \( \psi \coloneqq - \varphi \);
we define
\( p_2 \coloneqq \delta' \), i.e.\ the dual map of \(\delta \).
The Lie bracket between
two elements in \( \fg[\![x]\!]' \) is given by the sum \( p_1 + p_2 \). Defining
\( [a, f] \) as in \cref{eq:unique_commutator_on_L}
the evaluation form becomes invariant
and we get a topological Manin pair
\( (L, \fg[\![x]\!]) \)
with the Lagrangian subspace \( \fg[\![x]\!]'\).
These constructions are clearly inverse to each other.
\end{proof}

Combining the classification
of Manin pairs mentioned in
\cref{sec:top_manin_pairs}
with
\cref{cor:series_and_subspaces,lem:M_pair_q_bial_correspondance}
we get the following 
description of all topological
quasi-Lie bialgebra structures
on \( \fg[\![x]\!] \).

\begin{lemma}%
\label{lem:bijection_skew_Lagrangian}
There is a bijection 
between 
topological quasi-Lie
bialgebra structures on \( \fg[\![x]\!] \)
and Lagrangian
subspaces \( W \subset L(n, \alpha) \) or \( L(\infty) \) complementary to the diagonal \( \Delta \), 
where \( \alpha = (\alpha_i \in F \mid -\infty < i \le n-2) \) is an arbitrary sequence and \( n \ge 0 \).
Moreover, such 
Lagrangian subspaces
\( W \subset L(n, \alpha) \)
are in bijection with skew-symmetric sequences of type \( (n, 1/(x^n 
\alpha(x))) \).
\end{lemma}

In view of this result
we call skew-symmetric series of
type \( (n,s) \) as well as their
projections onto the first component
\emph{quasi-\(r\)-matrices}. 
Quasi-Lie bialgebra structures can also be described using their associated quasi-\(r\)-matrices in the following way.

\begin{proposition}%
\label{lem:CYB(r)_phi_delta_r}
Assume \( (\fg[\![x]\!], \delta, \varphi) \) is 
a topological quasi-Lie bialgebra
and let \( r \in L_2(n, \alpha) \) be the corresponding quasi-\(r\)-matrix given by the bijection from \cref{lem:bijection_skew_Lagrangian}.
Under the identification 
\( \fg[\![(x,[x])]\!] \cong \fg[\![x]\!] \)
we have the following identities:
\begin{itemize}
    \item \( [a \ot 1 + 1 \ot a, r] = -\delta(a) \)
    for any \( a \in \fg[\![x]\!] \) and
    \item \( \textnormal{CYB}(r) =-\varphi. \)
\end{itemize}
The same is true for the projection
\( r \in (\fg \ot \fg)(\!(x)\!)[\![y]\!] \).
\end{proposition}
\begin{proof}
We start, as in the proof of
\cref{lem:M_pair_q_bial_correspondance},
by fixing an identification
\( L(n, \alpha) = \Delta \add W(r) \cong \fg[\![x]\!] \add \fg[\![x]\!]' \).
Let \( \{ v_{k,i} \} \) be a basis
for \( \fg[\![x]\!]' \)
dual to \( \{\varepsilon_{k,i} \coloneqq b_i y^k \} \).
Then \( r = \sum_{k \ge 0} \sum_{i = 1}^d v_{k,i} \ot \varepsilon_{k,i} \) and we have
\begin{equation*}
\begin{aligned}
    [a \ot 1 + 1 \ot a, r] &= 
    \sum_{k \ge 0} \sum_{i=1}^d
    [a, v_{k,i}] \ot \varepsilon_{k,i} + v_{k,i} \ot [a, \varepsilon_{k,i}] \\
    &= \sum_{k \ge 0} \sum_{i=1}^d
    (-v_{k,i} \circ \ad_a + (v_{k,i} \ot 1)\delta(a)) \ot \varepsilon_{k,i} + v_{k,i} \ot [a, \varepsilon_{k,i}].
\end{aligned}    
\end{equation*}
Applying \( \langle v_{\ell,j} \ot v_{m, t}, - \rangle\) to the equality above
we get
\begin{equation*}
\begin{aligned}
\langle v_{\ell,j} \ot v_{m, t}, [a \ot 1 + 1 \ot a, r] \rangle &= 
    \sum_{k \ge 0} \sum_{i=1}^d
    \langle v_{\ell,j} \ot v_{m, t}, (v_{k,i} \ot 1)\delta(a) \ot \varepsilon_{k,i}  \rangle \\
    &= \langle v_{\ell,j}, (v_{m,t} \ot 1)\delta(a)  \rangle \\
    &= \langle v_{\ell, j} \ot v_{m,t}, -\delta(a) \rangle.
\end{aligned}
\end{equation*}
Applying instead 
\( \langle \varepsilon_{\ell,j} \ot v_{m, t}, - \rangle\) to the same equality we obtain
\begin{equation*}
\begin{aligned}
\langle \varepsilon_{\ell,j} \ot v_{m, t}, [a \ot 1 + 1 \ot a, r] \rangle &= 
    \sum_{k \ge 0} \sum_{i=1}^d
    \langle \varepsilon_{\ell,j} \ot v_{m, t}, (-v_{k,i} \circ \ad_a) \ot \varepsilon_{k,i} + v_{k,i} \ot [a, \varepsilon_{k,i}]  \rangle \\
    &= - \langle \varepsilon_{\ell,j}, v_{m,t} \circ \ad_a \rangle
    + \langle v_{m,t}, [a, \varepsilon_{\ell,j}] \rangle \\
    &= 0.
\end{aligned}
\end{equation*}
This implies the desired equality
\( [a \ot 1 + 1 \ot a, r] = -\delta(a) \). The identity \(\textnormal{CYB}(r) = -\varphi\) follows from the skew-symmetry of \(r\), \cref{thm:series_and_subspaces} and the fact that \(\varphi = -\psi\) according to the proof of \cref{lem:M_pair_q_bial_correspondance}.
\end{proof}

\begin{remark}
Assume \( r \in (\fg \ot \fg)(\!(x)\!)[\![y]\!] \) 
is a series such that 
\begin{equation}%
\label{eq:condition_on_delta_series}
    [f(x) \ot 1 + 1 \ot f(y), r(x,y)] \in (\fg \ot \fg)[\![x,y]\!]
\end{equation}
for all \( f \in \fg[\![x]\!] \).
Write \( r = s(x^{-1},y) + g(x,y) \), where
\( s \in x^{-1}(\fg \ot \fg)[x^{-1}][\![y]\!] \)
and \( g \in (\fg \ot \fg)[\![x,y]\!] \).
Then, because of \cref{eq:condition_on_delta_series},
we must have
\begin{equation*}
    [a \ot 1 + 1 \ot a, s(x^{-1}, y)] = 0
\end{equation*}
for all \( a \in \fg \). Since the \( \fg \)-invariant elements 
of \( \fg \ot \fg \) are precisely the multiples of the
quadratic Casimir element \( \Omega \), we have the identity
\( s(x^{-1}, y) = p(x^{-1},y) \Omega \) for some
\( p \in x^{-1}F[x^{-1}][\![y]\!] \).
Furthermore, the condition
\begin{equation*}
    [ax \ot 1 + 1 \ot ay, p(x^{-1}, y) \Omega] = 
    [a(x-y) \ot 1, p(x^{-1}, y) \Omega] \in (\fg \ot \fg)[\![x,y]\!]
\end{equation*}
implies  \( (x-y)p(x^{-1}, y) \in F[\![x,y]\!] \),
meaning that there exists an \(s \in F[\![y]\!] \)
such that \( p(x^{-1}, y) = s(y) / (x-y) \).
In other words, \( r \) has the form
\cref{eq:standard_form}.
This result can be considered as another motivation
to study series of type \( (n,s) \).
\end{remark}


Observe that if we know one Lagrangian
subspace \( W_0 \) inside \( L \cong \fg[\![x]\!] \add \fg[\![x]\!]' \)
then any other Lagrangian subspace can be
constructed from \( W_0 \) through twisting.
More precisely, if 
\( s = \sum_{i} a_i \ot b^i \in (\fg \ot \fg)[\![x,y]\!] \)
is a skew-symmetric tensor,
then we can associate with it a (twisted) Lagrangian subspace
\begin{equation}%
\label{eq:twist_by_s_of_W}
    W_s \coloneqq \left\{ \sum_{i} B(b^i, w)a_i - w \mid w \in W  \right\} \subseteq L
\end{equation} 
complementary to \( \fg[\![x]\!] \).
The converse is also true; for proof see
\cite{Raschid_Stepan_classical_twists}.
In other words, the following statement holds.

\begin{lemma}%
\label{lem:twists_subspaces}
There is a bijection between Lagrangian subspaces \( W \subseteq L(n, \alpha)\) or \( L(\infty) \) and skew-symmetric tensors in \( (\fg \ot \fg )[\![x,y]\!]\).
\end{lemma}

Combining \cref{lem:CYB(r)_phi_delta_r}, \cref{eq:twist_by_s_of_W}
and the algorithm for constructing a quasi-\(r\)-matrix
from a Lagrangian subspace \( W \subset L(n, \alpha) \), \( W \add \Delta = L(n, \alpha) \),
we obtain the following twisting rules for Lagrangian subspaces, quasi-Lie bialgebra structures and quasi-\(r\)-matrices.

\begin{lemma}
\label{lem:twists_subspaces2}
Let \( (\fg[\![x]\!], \delta, \varphi ) \)
be a topological quasi-Lie bialgebra
structure
corresponding to the quasi-\(r\)-matrix \(r\).
If we twist \( W(r) \) with a skew-symmetric tensor
\(s\) as described in 
\cref{eq:twist_by_s_of_W} we obtain
another topological
quasi-Lie bialgebra
\( (\fg[\![x]\!], \delta_s, \varphi_s) \), such that
\begin{enumerate}
    \item \( W(r)_s = W(r-s) \);
    \item \( \delta_s = \delta + ds \);
    \item \(\varphi_s = \varphi + \textnormal{CYB}(s) - \frac{1}{2}\textnormal{Alt}((\delta \ot 1)s) \),
\end{enumerate}
where \( ds(a) \coloneqq [a \ot 1 + 1 \ot a, s] \).
\end{lemma}

\begin{remark}
Since any quasi-\(r\)-matrix \(r \)
defines a topological quasi-Lie
bialgebra structure \( \delta(a) = [a \ot 1 + 1 \ot a, r]\)
on \( \fg[\![x]\!] \), the third
condition in \cref{def:top_quasi_lie_bialg}
is trivially satisfied. In other words,
\[ \textnormal{Alt}((\delta \ot 1 \ot 1)\textnormal{CYB}(r)) = 0\]
for any quasi-\(r\)-matrix \( r \).
\end{remark}

\Cref{lem:twists_subspaces} and \cref{lem:twists_subspaces2} state that, in order to obtain a description
of topological quasi-Lie bialgebra
structures on \( \fg[\![x]\!] \) 
up to twisting
it is enough to find a single Lagrangian subspace \( W_0 \), complementary to \( \fg[\![x]\!] \),
inside \( L(\infty) \) and 
each \( L(n, \alpha) \). The same is true for the associated quasi-\(r\)-matrices

The case \( L(\infty) \) is trivial,
because by definition 
\(\fg[\![x]\!]' = \bigoplus_{j \ge 0} \fg \ot a_j \subseteq L(\infty) \) 
is a Lagrangian subalgebra (see \cref{def:trivial_double}).
Similar to the Lie bialgebra case,
topological quasi-Lie bialgebras
corresponding to the Manin pair
\( (L(\infty), \fg[\![x]\!]) \) are called
\emph{degenerate}.

Let us now focus on
\emph{non-degenerate} topological
quasi-Lie bialgebra structures,
i.e.\ the ones corresponding to the
Manin pairs \( (L(n, \alpha), \Delta) \).
By \cref{prop:Manin_pairs}
for each Manin pair 
\( (L(n, \alpha), \Delta) \) 
there exists an appropriate coordinate
transformation that makes it into 
\( (L(n, \beta), \Delta) \),
where \( \beta_0 = \alpha_0 \)
and all other \( \beta_i = 0\).
This means, that to classify all
non-degenerate topological quasi-Lie bialgebras on \( \fg[\![x]\!] \),
up to coordinate transformations and twisting,
it is enough to construct a Lagrangian
subspace \( W_0 \) within each
\( L(n, \alpha_0) \coloneqq 
L(n, (\dots, 0, \alpha_0, 0, \dots, 0))\) complementary to \( \Delta \).
Equivalently, it is enough 
to find a quasi-\(r\)-matrix
of type \( (n, \alpha_0) \)
for any \( n \ge 0 \) and \( \alpha_0 \in F \).


\subsection{Lagrangian subspaces of \( L(n, \alpha_0) \)}%
\label{sec:Lagrangian_subspaces}
As before we let
\( \{ b_i \}_{i=1}^d \)
be an orthonormal basis for \( \fg \)
with respect to the Killing form \( \kf \).
The form \( B \) on \( L(n, \alpha_0) \)
has the following explicit form
\begin{equation}
  B(a \ot (f, [p]), b \ot (g, [q]))
  =
  \begin{cases}
  \kf(a,b)\left\{ \textnormal{coeff}_{n-1}(fg - pq)
  - \alpha_0 \textnormal{coeff}_{0}(fg - pq)\right\} & \, \text{ if } \, n \ge 2, \\
  \kf(a,b)\textnormal{coeff}_{n-1}(fg - pq) & \, \text{ if } \, n = 0, 1.
  \end{cases}
\end{equation}
We now present an explicit construction for a
Lagrangian subspace of \( L(n, \alpha_0) \)
complementary to \( \Delta \)
for arbitrary \( n \ge 0 \) and \( \alpha_0 \in F\). Using the twisting procedure from \cref{lem:twists_subspaces2}, this subspace can be twisted in order to obtain all other Lagrangian subspaces of \(L(n,\alpha_0)\) complementary to \(\Delta\).
\paragraph{n = 0:}
When \( n = 0 \), the subalgebra
\( W_0 \coloneqq x^{-1}\fg[\![x^{-1}]\!]
\subseteq \fg(\!(x)\!) \)
is known to be Lagrangian.


\paragraph{n = 1:} For \( n = 1 \) it is easy to see
that the subspace
\begin{equation}
 W_0 \coloneqq \textnormal{span}_F\{ b_i(1,-1), b_i(x^{-k}, 0) \mid k \ge 1, \, 1 \le i \le d \} \subset L(1, \alpha_0)
\end{equation}
is Lagrangian and complementary to the diagonal \( \Delta \).

\paragraph{n = 2k:}
For even \( n \ge 2 \) and arbitrary 
\( \alpha_0 \in F \) the subspace
\( W_0 \subset L(n, \alpha_0) \)
spanned by the elements
\begin{equation*}
\begin{aligned}
  b_i &\left\{ 
  (x^{(n-1)-m}, 0) - \alpha_0 (x^{2(n-1) - m},0)
  + \alpha_0^2 (x^{3(n-1) - m},0) - \alpha_0^3
  (x^{4(n-1) - m},0) + \dots
  \right\}, \ 0 \le m \le \frac{n}{2}-1, \\
  b_i &(
  0, -[x]^{(n-1)-\ell}
  ), \ \frac{n}{2} \le \ell < n-1, \\
  b_i &( 
  0, -1 + \frac{\alpha_0}{2}[x]^{n-1}
  ), \\
  b_i &(x^{-k}, 0), k \ge 1,
  \phantom{\frac{\alpha_0}{2}}
\end{aligned}
\end{equation*}
is Lagrangian and complementary
to the diagonal.
\paragraph{n = 2k + 1:}
Modifying slightly the basis for even case we obtain the following basis
for \( W_0 \subset L(n, \alpha_0) \)
with odd \( n \ge 3 \):
\begin{equation*}
\begin{aligned}
  b_i & \left\{ 
  (x^{(n-1)-m}, 0) - \alpha_0 (x^{2(n-1) - m},0)
  + \alpha_0^2 (x^{3(n-1) - m},0) - \alpha_0^3
  (x^{4(n-1) - m},0) + \dots
  \right\}, \ 0 \le m \le \frac{n-1}{2}-1, \\
  b_i & \left\{ 
  (x^{\frac{n-1}{2}}, -[x]^{\frac{n-1}{2}})
  - \alpha_0(x^{\frac{3(n-1)}{2}}, 0) + \alpha_0^2(x^{\frac{5(n-1)}{2}}, 0) - \alpha_0^3(x^{\frac{7(n-1)}{2}}, 0) + \dots
  \right\}, \\
  b_i & (
  0, -[x]^{(n-1)-\ell}
  ), \ \frac{n-1}{2}+1 \le \ell < n-1, \\
  b_i & ( 
  0, -1 + \frac{\alpha_0}{2}[x]^{n-1}
  ), \\
  b_i & (x^{-k}, 0), k \ge 1.\phantom{\frac{\alpha_0}{2}}
\end{aligned}
\end{equation*}

The subspaces above
were constructed by ''guessing''.
However, there is an abstract
procedure that produces Lagrangian
subspaces for arbitrary
\( n \) and \(\alpha\).
We present it here for completeness.

The easiest skew-symmetric \((n,s)\)-type series is given by
\begin{align*}
    r(x,y) &\coloneqq
    \frac{1}{2}\left(\frac{s(x)y^n\Omega}{x-y} + \frac{s(y)x^n\Omega}{x-y}\right) 
    = \frac{s(x)y^n\Omega}{x-y} + \frac{\Omega}{2}\left(\frac{s(y)x^n-s(x)y^n}{x-y}\right) \\
    &= 
    \frac{s(x)y^n\Omega}{x-y} - \frac{1}{2}\sum_{k,\ell = 0}^\infty \sum_{i,j = 1}^d B(sw_{k,i},sw_{\ell,j})b_i(x,[x])^k \otimes b_j(y,[y])^\ell,
\end{align*}
where we recall that
\begin{equation*}
\begin{aligned}
B(s w_{k,i},s w_{\ell,j}) =
\begin{cases}
- s_{k + \ell - n + 1} & \textnormal{ if } \, i = j, \, 0 \le k,\ell \le n-1 \, \textnormal{ and } \, k + \ell \ge n-1, \\
s_{k + \ell - n + 1} & \textnormal{ if } \, i = j \, \textnormal{ and }k \, ,\ell \ge n, \\
0 & \textnormal{ otherwise}.
\end{cases}
\end{aligned}
\end{equation*}
By \cref{cor:series_and_subspaces}
the subspace
\begin{equation*}
    \begin{split}
        W(r) &= \textnormal{span}_F\left\{sw_{k,i} - \frac{1}{2}\sum_{\ell = 0}^\infty B(sw_{\ell,i},sw_{k,i})b_i(x,[x])^\ell\, \bigg| \, k \ge 0, 1 \le d \le n\right\} 
        \\&= \textnormal{span}_F\left\{sw_{k,i} + \frac{1}{2}\left(\sum_{\ell = 0}^{n-1}s_{k+\ell - n 
       + 1}b_i(x,[x])^\ell - \sum_{\ell = n}^\infty s_{k+\ell - n 
        +1}b_i(x,[x])^\ell \right)\, \bigg| \, k \ge 0, 1 \le d \le n\right\} 
    \end{split}
\end{equation*}
is Lagrangian and complementary to the diagonal. Here we used the convention that \(s_k = 0\) for \(k < 0\). 
Calculating the basis
explicitly for some
particular \(s\)
requires some
effort and it may not look
as friendly as the ones given above.


\subsection{Quasi-\(r\)-matrices}%
\label{sec:corresponding_quasi_r_matr}
The goal of this section is to describe the
quasi-\(r\)-matrices corresponding to the Lagrangian subspaces described in the previous
section. The twisting procedure from \cref{lem:twists_subspaces2} then yields all other quasi-\(r\)-matrices.

The proof of \cref{thm:series_and_subspaces}
gives us an algorithm for constructing
a series of type
\( \left(n, s(x) \coloneqq 1/(x^n \alpha(x)) \right) \)
from a subspace \( W \subset L(n, \alpha) \)
complementary to the diagonal. More precisely,
the desired series is given by
\begin{equation}
    \sum_{k \ge 0} \sum_{i = 1}^d
    v_{k,i} \ot b_i (y^k, [y]^k),
\end{equation}
where 
\begin{equation*}
    W = \textnormal{span}_F \{ v_{k,i} \mid k \ge 0, \, 1 \le i \le d \}
    \, \text{ and } \, B(v_{k,i}, b_j (y^{\ell},[y]^\ell)) = \delta_{i,j}\delta_{k,\ell},
\end{equation*}
i.e.\
\( \{ v_{k,i} \} \) is a basis of \( V \)
dual to \( \{b_i (y^k,[y]^k) \} \).
Indeed, non-degeneracy of the form \( B \) then
implies that \( v_{k,i} \) has the desired form
\( v_{k,i}= sw_{k,i} + g_{k,i} \)
for some \( g_{k,i} \in (\fg \ot \fg)[\![x,y]\!] \). 

Applying this idea to \( W_0 \)'s constructed
in the preceding section we get the following series.

\paragraph{n = 0:}
The classical \(r\)-matrix (equivalently \( (0,1)\)-type series)
corresponding to
\( W_0 \coloneqq x^{-1}\fg[\![x^{-1}]\!]
\subseteq \fg(\!(x)\!) \)
is the Yang's matrix \( \Omega/(x-y)\).

\paragraph{n = 1:} The quasi-\(r\)-matrix
corresponding to \( \textnormal{span}_F\{ b_i(1,-1), b_i(x^{-k}, 0) \mid k \ge 1, \, 1 \le i \le d \} \subset L(1, \alpha_0) \)
is
\begin{equation*}
    \frac{y \Omega}{x-y} + \frac{1}{2} \sum_{i=1}^d b_i (1, -1) \ot b_i (1,1) \in L_2(1, 1)
    \, \text{ with the projection } \,
    \frac{y \Omega}{x-y} + \frac{1}{2} \Omega \in (\fg \ot \fg)(\!(x)\!)[\![y]\!].
\end{equation*}

\paragraph{n = 2k:}
For even \( n \ge 2 \) and arbitrary 
\( \alpha_0 \in F \) we have the following
quasi-\(r\)-matrix
\begin{equation*}
\begin{aligned}
    \frac{1}{1 + \alpha_0 x^{n-1}}
    \frac{y^n \Omega}{x-y} 
    &+
    \frac{\Omega}{1+\alpha_0 x^{n-1}}
    \sum_{0 \le m < \frac{n}{2}} x^{(n-1)-m}y^m \\
    &+ \frac{\alpha_0 \Omega}{(1+\alpha_0x^{n-1})(1+\alpha_0y^{n-1})}
    \left( 
    y^{2(n-1)} + \sum_{\frac{n}{2} \le \ell < n-1} x^{(n-1)-\ell}y^{(n-1)+\ell}
    - \frac{1}{2}x^{n-1}y^{n-1}
    \right).
\end{aligned}
\end{equation*}

\paragraph{n = 2k + 1:}
In the odd case \( n \ge 3 \) the series 
corresponding to \( W_0 \subset L(n, \alpha_0) \) is
\begin{equation*}
\begin{aligned}
    \frac{1}{1 + \alpha_0 x^{n-1}}
    \frac{y^n \Omega}{x-y} 
    &+
    \frac{\Omega}{1+\alpha_0 x^{n-1}}
    \left( x^{\frac{n-1}{2}}y^{\frac{n-1}{2}} +
    \sum_{0 \le m < \frac{n-1}{2}} x^{(n-1)-m}y^m \right) \\
    &+ \frac{\alpha_0 \Omega}{(1+\alpha_0x^{n-1})(1+\alpha_0y^{n-1})}
    \left( 
    y^{2(n-1)} + \sum_{\frac{n-1}{2} < \ell < n-1} x^{(n-1)-\ell}y^{(n-1)+\ell}
    - \frac{1}{2}x^{n-1}y^{n-1}
    \right).
\end{aligned}
\end{equation*}

\section{Lie algebra splittings of \(L(n,\alpha)\) and generalized \(r\)-matrices}%
\label{sec:algebra_splittings}
By \cref{cor:series_and_subspaces}
we have a bijection between
subalgebras of \( L(n, \alpha) \)
and series of type \( (n, 1/(x^n \alpha(x))) \) solving GCYBE.
Therefore, we can construct new
solutions to GCYBE by
finding subalgebras of \( L(n, \alpha) \) complementary to the diagonal.
However, 
as the following result shows, the
most interesting new solutions
should arise from
unbounded subalgebras of \( L(n,\alpha) \),
\( n > 2 \).

\begin{proposition}\label{prop:new_subalgs_are_bounded}
Let \(L(n,\alpha) = \Delta \dotplus W\) 
for some subalgebra \(W \subset L(n,\alpha)\).
Assume \( W \) is bounded, i.e.\ 
there is an integer \( N > 0 \)
such that
\begin{equation*}
    x^{-N}\fg[x^{-1}] \subseteq W_+ \subseteq x^N\fg[x^{-1}],
\end{equation*}
where \(W_+\) is the projection of 
\(W \subset L(n,\alpha) = \fg(\!(x)\!) \oplus \fg[x]/x^n\fg[x]\) 
on the first component \(\fg(\!(x)\!)\). 
Then there is an element \(\sigma \in  \textnormal{Aut}_{F[x]-\textnormal{LieAlg}}(\fg[x])\) such that
\[\{0\} \times [x^2]\fg[x]/x^n\fg[x] \subseteq (\sigma \times \sigma)W \subseteq x\fg[x^{-1}] \times \fg[x]/x^n\fg[x]\] 
and the image \(\widetilde{W}\) under the canonical projection \(L(n,\alpha) \to L(2,\alpha)\) is a subalgebra satisfying \(L(2,\alpha) = \Delta \dotplus \widetilde{W}\).

In the language of \( (n,s)\)-type
series: Let
\[r = \frac{s(x) y^n \Omega}{x-y} + g(x,y)\] 
be the generalized \(r\)-matrix
corresponding to a bounded \( W \subset L(n,\alpha), n \ge 2 \).
Then there is \(p(x,y) \in (\fg \otimes \fg)[x,y]\) of degree at most one in x and an element
\(\sigma \in  \textnormal{Aut}_{F[x]-\textnormal{LieAlg}}(\fg[x])\) such that
\begin{equation*}
    (\sigma(x) \ot \sigma(y))r(x,y) = y^{n-2} \Big(\underbrace{ \frac{s(x)y^2 \Omega}{x-y} + p(x,y)}_{r'(x,y)} \Big),
\end{equation*}
where \( r' \)
is a generalized \(r\)-matrix
in \( L_2(2,\alpha) \).
\end{proposition}
\begin{proof}
The condition 
\(x^{-N}\fg[x^{-1}] \subseteq W_+ \subseteq x^N\fg[x^{-1}]\) 
means exactly that \(W_+\) is an order.
Moreover, since \( W \) is complementary
to the diagonal, we have
\( W_+ + \fg[x] = \fg[x,x^{-1}] \).
It was shown in \cite{Stolin_rational} that
such orders, up to the action of
some
\(\sigma \in \textnormal{Aut}_{F[x]-\textnormal{LieAlg}}(\fg[x]) \),  are
contained in a maximal order
\(\mathfrak{M}\) associated to the
so called fundamental simplex
\( \Delta_{\textnormal{st}}\). These maximal orders are explicitly described in \cite{Stolin_rational} and satisfy \(\mathfrak{M} \subseteq x\fg[x^{-1}]\).
Therefore, we have
\(\sigma W_+ \subseteq \mathfrak{M} \subseteq x\fg[x^{-1}]\).
Moreover, we have the identity
\begin{equation*}
    (\sigma \times \sigma)W \add \Delta = L(n,\alpha),
\end{equation*}
implying the inclusion
\(\{0\} \times [x^2]\fg[x]/x^n\fg[x] \subseteq (\sigma \times \sigma)W\). The remaining parts follow straightforward from the construction \cref{thm:series_and_subspaces}.
\end{proof}

Unfortunately, we have not found
a new example of an unbounded
subalgebra of \( L(n,\alpha) \).
However, we present an infinite
family of bounded subalgebras.
We believe these examples are still
interesting because their orthogonal
complements, which are important
in the view of
Adler-Kostant-Symes scheme,
are unbounded if \(\alpha \neq 0\).

Consider the subspaces of \( L(n,\alpha_0) \), \( n > 0 \):
\begin{equation*}
    \begin{aligned}
    W_0 &= \textnormal{span}_F \{ b_i(x^{-k},0), b_i(1,0), b_i(0, -[x]^\ell) \mid k \ge 1, \, 1 \le \ell \le n-1 \}, \\
    W_1 &= \textnormal{span}_F \{ b_i(x^{-k},0), b_i(0,-1), b_i(0, -[x]^\ell) \mid k \ge 1, \, 1 \le \ell \le n-1 \}.
    \end{aligned}
\end{equation*}
These are clearly subalgebras.
The corresponding generalizerd
\(r\)-matrices are
\begin{equation*}
\begin{aligned}
    r_0 &= \frac{1}{1+\alpha_0 x^{n-1}} \frac{y^n\Omega}{x-y} + \frac{y^{n-1} \Omega}{(1 + \alpha_0 x^{n-1})(1 + \alpha_0 y^{n-1})}  \\
    & \ \ \  +\frac{\alpha_0 \Omega}{(1 + \alpha_0 x^{n-1})(1 + \alpha_0 y^{n-1})}\left( y^{2(n-1)} + \sum_{ 0 \le \ell < n-1} x^{(n-1)-\ell} y^{(n-1)+\ell} \right) \\
    &= \frac{y^{n-1}}{1+\alpha_0 y^{n-1}} \left(\frac{y\Omega}{x-y} + \Omega\right),\\
    r_1 &= \frac{1}{1+\alpha_0 x^{n-1}} \frac{y^n\Omega}{x-y} 
    +\frac{\alpha_0 \Omega}{(1 + \alpha_0 x^{n-1})(1 + \alpha_0 y^{n-1})}\left( y^{2(n-1)} + \sum_{ 0 < \ell < n-1} x^{(n-1)-\ell} y^{(n-1)+\ell} \right) \\ & = \frac{1}{1+\alpha_0 y^{n-1}}\frac{y^n\Omega}{x-y}. 
\end{aligned}
\end{equation*}

By considering 
decompositions
\( \fg = \fs_1 \add \fs_2 \) of
\( \fg \)
into direct sums of subalgebras
we can get an infinite
family of generalized \(r\)-matrices
''in between'' \(r_0\) and \(r_1\).
More precisely,
let \( \{s_{1,i}\}_{i=1}^{d_1}  \)
and \( \{s_{2,j}\}_{j=1}^{d_2}  \)
be bases for 
\( \fs_1 \) and \( \fs_2 \)
respectively.
Such a decomposition leads
to another subalgera of \( L(n, \alpha_0) \):
\begin{equation*}
\begin{aligned}
    W_{01} \coloneqq \textnormal{span}_F \Big\{ b_i(x^{-k},0), s_{1,m}(1,0),
    s_{2,j}(0,1), b_i(0, -[x]^\ell) \mid k \ge 1, \, &1 \le \ell \le n-1, \, 1 \le i \le d, \\
    &1 \le m \le d_1, \, 1 \le j \le d_2 \Big\}.
\end{aligned}
\end{equation*}
Rewrite the elements \( b_i \)
in terms of \( s_{1,m} \)
and \( s_{2,j} \):
\begin{equation*}
    b_i = \sum_{m=1}^{d_1} \lambda_{1,m}^i s_{1,m}
    + \sum_{j=1}^{d_2}
    \lambda_{2,j}^i s_{2,j},
\end{equation*}
where \( \lambda_{1,m}^i, \lambda_{2,j}^i \in F \).
Finding a basis in \( W_{12} \)
dual to \( \{ b_i(y^m, [y]^m) \} \subset \Delta \)
and then projecting the generating
series for \( W_{01} \) onto
the first component we obtain
the following generalized 
\( r \)-matrix
\begin{equation}%
\label{eq:generalized_nonzero_alpha}
\begin{aligned}
    r_{01} &= \frac{1}{1+\alpha_0 x^{n-1}} \frac{y^n\Omega}{x-y} 
    +\frac{\alpha_0 \Omega}{(1 + \alpha_0 x^{n-1})(1 + \alpha_0 y^{n-1})}\left( y^{2(n-1)} + \sum_{ 0 < \ell < n-1} x^{(n-1)-\ell} y^{(n-1)+\ell} \right) \\
    &+
    \frac{y^{n-1}}{1 + \alpha_0 y^{n-1}}
    \sum_{i=1}^d \sum_{m=1}^{d_1}
    \lambda^i_{1,m} s_{1,m} \ot b_i.
    \\&= \frac{y^{n-1}}{1+\alpha_0y^{n-1}}\left(\frac{y\Omega}{x-y} + \sum_{i=1}^d \sum_{m=1}^{d_1}
    \lambda^i_{1,m} s_{1,m} \ot b_i\right)
\end{aligned}
\end{equation}
Clearly \(r_{01}\) coincides with
\( r_0 \) when \( \fs_1 = \fg \)
and \( r_1 \) if \( \fs_2 = \fg \).
The corresponding
orhogonal complements are
\begin{equation}
\begin{aligned}
    W_0^{\perp} &= W(\overline{r_0}) = 
    \textnormal{span}_{F}
    \left\{
    b_i\left(0,[x]^{n-1}\right),
    b_i\left(\frac{x^{-k(n-1)-m}}{1+\alpha_0 x^{n-1}}, 0\right)
     \mid k \ge -1, 0 < m < n-1 \right\}, \\ ~\\
    W_1^{\perp} &= W(\overline{r_1}) = 
    \textnormal{span}_{F}
    \left\{
    b_i\left(\frac{x^{-k(n-1)-m}}{1+\alpha_0 x^{n-1}}, 0\right)
     \mid k \ge -1, 0 \le m < n-1 \right\}, \\~\\
     W_{01}^{\perp} &= W(\overline{r_{01}}) = 
     \mathfrak{s}_1^\perp \left(\frac{x^{n-1}}{1+\alpha_0 x^{n-1}},0\right)
     \add 
     \mathfrak{s}_2^\perp(0, [x]^{n-1}) \\
     & \phantom{W(\overline{r_1}) = W(\overline{r_{01}}}  \add
    \textnormal{span}_{F}
    \left\{
    b_i\left(\frac{x^{-k(n-1)-m}}{1+\alpha_0 x^{n-1}}, 0\right)
     \mid k \ge -1, 0 < m < n-1 \right\},
\end{aligned}
\end{equation}
which are unbounded because of the factor \( 1/ (1+\alpha_0x^{n-1}) \).

Note that a series of
type \( (n,s) \)
defines a subspace inside
\( L(n,\alpha ) \) for any
\( \alpha \), because the 
subalgebra property is not
affected by the form.
With the previous examples in mind we can prove the following statement.

\begin{lemma}
Let \(B_0 \) and \( B_\alpha \) be the bilinear forms
on \( L(n, 0) \) and \( L(n, \alpha) \) respectively.
For a series \(r\) of type \( (n,s) \) we have
\begin{equation}
    W\left(r\right)^{\perp_{B_\alpha}} = \frac{1}{x^n \alpha(x)} W(r)^{\perp_{B_0}} \subset L(n, \alpha).
\end{equation}
\end{lemma}
\begin{proof}
Set \( u(x) \coloneqq 1/(x^n \alpha(x)) \). Write
\[ r = \sum_{k \ge 0} \sum_{i = 1}^d (s w_{k,i} + g_{k,i}) \ot b_i(y^k,[y]^k)  \ \text{ and } \  \overline{r} = \sum_{k \ge 0} \sum_{i = 1}^d (sw_{k,i} + \overline{g_{k,i}}) \ot b_i(y^k,[y]^k).  \]
Then by \cref{thm:series_and_subspaces} and definition \cref{eq:form_L_n_alpha} 
\(B_{\alpha}(sw_{k,i} + g_{k,i}, u( sw_{\ell,j} + \overline{g_{\ell,j}}))
= B_0 (sw_{k,i} + g_{k,i}, sw_{\ell,j} + \overline{g_{\ell,j}}) = 0\)
\end{proof}

\newpage
\printbibliography\end{document}